\documentclass[onethmnum,onefignum,final]{siamltex}
\usepackage[english]{babel}
\usepackage{amsfonts}
\usepackage{amsmath}
\usepackage{amssymb}
\usepackage{bbm}
\usepackage{epsfig}
\usepackage{color}
\usepackage{enumerate}
\usepackage{showkeys}
\usepackage{dsfont}
\usepackage{mathrsfs}
\renewcommand{\Re}{\operatorname*{Re}}
\renewcommand{\Im}{\operatorname*{Im}}
\newcommand{\real}{\operatorname*{Re}}

\newcommand{\vrho}{\varrho}
\newcommand{\tol}{\mathrm{tol}}
\newcommand{\tf}{\mathrm{f}}
\newcommand{\te}{\mathrm{e}}

\def\mi{{\mathrm{i}\mkern1mu}}
\def\opi{{\mathrm{i}\mkern1mu}}
\def\erf{{\mathrm{erf}\mkern1mu}}
\def\bone{\mathbbm{1}}
\def \e{{\operatorname*{e}}}
\def \veps{\varepsilon}
\def \z{\zeta}

\def \dt{h}
\newcommand{\qv}{{\mathbf{q}}}
\newcommand{\Qv}{{\mathbf{Q}}}

\newcommand{\tv}{{\mathbf{t}}}
\newcommand{\Vv}{{\mathbf{V}}}
\def \bv{{\mathbf{b}}}
\def \cv{{\mathbf{c}}}
\def \Wv{{\mathbf{W}}}
\newcommand{\omegav}{\boldsymbol{\omega}}
\newcommand{\tauv}{\boldsymbol{\tau}}

\newcommand{\gv}{{\mathbf{g}}}
\newcommand{\ev}{{\mathbf{e}}}
\newcommand{\Ev}{{\mathbf{E}}}
\newcommand{\Iv}{{\mathbf{I}}}
\newcommand{\fv}{{\mathbf{f}}}
\newcommand{\uv}{{\mathbf{u}}}
\newcommand{\Av}{{\mathbf{A}}}

\newcommand{\bR}{\mathbb{R}}

\newcommand{\IR}{\mathbb{R}}

 \newtheorem{remark}[theorem]{Remark}
\newtheorem{assumption}[theorem]{Assumption}

\makeatletter\@addtoreset{equation}{section}\makeatother
\graphicspath{ {images/}{images}}

\begin{document}

\title{Numerical approximation of the Schr\"odinger equation with concentrated potential}

\author{L. Banjai \thanks{The Maxwell Institute for Mathematical Sciences, School of Mathematical \& Computer Sciences, Heriot-Watt University, Edinburgh EH14 4AS, UK.
({\tt l.banjai@hw.ac.uk})} \and M. L\'opez-Fern\'andez \thanks{Department of Mathematical Analysis, Statistics and O.R. and Applied Mathematics, Faculty of Sciences, University of M\'alaga, Spain ({\tt maria.lopezf@uma.es}) and Department of Mathematics Guido Castelnuovo, Sapienza University of Rome, Italy}}

\maketitle

\begin{abstract}
  We present a family of algorithms for the numerical approximation of the Schr\"odinger equation with potential concentrated at a finite set of points. Our methods belong to the so-called fast and oblivious convolution quadrature algorithms. These algorithms are special implementations of Lubich's Convolution Quadrature which allow, for certain applications in particular parabolic problems, to significantly reduce the computational cost and memory requirements. Recently it has been noticed that their use can be extended to some hyperbolic problems. Here we propose a new family of such efficient algorithms tailored to the features of the Green's function for Schr\"odinger equations. In this way, we are able to keep the computational cost and the storage requirements significantly below existing approaches. These features allow us to perform reliable numerical simulations for longer times even in cases where the solution becomes highly oscillatory or  seems to develop finite time blow-up. We illustrate our new algorithm with several numerical experiments.
\end{abstract}


\vspace{1em}

\section{Introduction}
We consider the efficient numerical approximation of Schr\"odinger equations with the potential concentrated at a finite set of points in dimension $D=1,2,3$. These problems can be formally described by the equation
\begin{equation}\label{conschrodinger}
\frac1\mi \psi_t(t,x)=\left(\Delta - \sum_{j=1}^{M} V_j \delta(x-x_j) \right) \psi(t,x), \qquad x\in \bR^{D},
\end{equation}
where  $\Delta$ denotes the Laplacian and the coupling factors $V_j$ may depend on $t$, $V_j=V_j(t)$, and/or on the value of $\psi$ at $(t,x_j)$, $V_j=V_j(\psi(t,x_j))$, $j=1,\dots,M$. These models have been used to describe different phenomena in solid state physics, optics and acoustics, and have been rigorously analyzed by several authors in the mathematical physics community, starting from the one-dimensional case \cite{AT01}, followed by the three-dimensional case \cite{AAFT03,AAFT04}, and more recently the two-dimensional case \cite{CCF17,CCT19}. The reformulation of these models  as $M$-dimensional systems of Volterra integral equations has proven to be very useful for the analysis of the existence, uniqueness, and regularity of solutions.  In this paper we address the efficient numerical approximation of such integral representations. Even with the reduction to a finite dimensional system of integral equations, the numerical approximation of \eqref{conschrodinger} can be quite delicate \cite{Negulescu}, particularly in the nonlinear case. We notice that the long-time behaviour of the solution to the nonlinear Schr\"odinger equation (NLS) with concentrated potential is not always well-understood and it might become highly oscillatory and even develop blow-up in finite time.

Let us recall that for the Schr\"odinger equation in the whole space
\begin{equation}
\frac1\mi \psi_t(t,x)-\Delta \psi(t,x)=0, \qquad x\in \bR^{D},
\end{equation}
with $D=1,2,3$, the Green's function is given by \cite[Equation (2.24)]{Economou}
\begin{equation}\label{green_schrodinger}
k(t,x) =  %\frac{e^{i\pi D/4}}{4\pi t}^{D/2}
 \mi \left(\frac1{4\pi \mi t}\right)^{D/2}
e^{i\|x\|^2/4t}, \qquad x\in \bR^D, t>0.
\end{equation}
The corresponding transfer operator, i.e., the Laplace transform of $k$
\[
K(z,x) = \int_0^\infty e^{-zt} k(t,x) dt, \qquad \Re z > 0,
\]  is given by
\begin{equation}\label{transfer_schrodinger}
K(z,x) = \left\{\begin{array}{ll}
\displaystyle \frac{1}{2\sqrt{z/i}}\exp\left(-|x|\sqrt{z/i} \right) \qquad & \text{for } D = 1, \\[1em]
\displaystyle \frac{1}{2\pi}K_0\left(\|x\|\sqrt{z/i} \right) \qquad & \text{for } D = 2,\\[1em]
\displaystyle \frac{1}{4\pi \|x\|}\exp \left(-\|x\|\sqrt{z/i} \right) \qquad & \text{for } D = 3.
\end{array}
\right.
\end{equation}
In the above $K_0$ is the modified Bessel function of the second kind \cite{NIST} and
\begin{equation}\label{def_sqrt}
\sqrt{z/i} = e^{-i\pi/4} \sqrt{z},
\end{equation}
where $\sqrt{z}$ is the branch with the positive real part  and with the branch cut along the negative real axis. Hence the transfer operator $K(x,z)$ as a function of $z$ is analytically extended to the cut complex plane $\mathbb{C} \setminus (-\infty,0]$.
% We will also use the notation
% $$
% d :=\|x\|,
% $$
% which will be treated as a parameter in what follows.

In applications, see \cite{Lopez_Castillo_1988,Negulescu} and Section~\ref{sec:numerics}, it is important to be able to compute accurately and efficiently the convolution in time with the Green's function. Namely,
\begin{equation}
  \label{eq:Sconv}
\psi(t,x) = \int_{-\infty}^t k(t-\tau,x-x_0) f(\tau, x_0) d\tau,
\end{equation}
solves the Schr\"{o}dinger equation with a source at $x_0$:
\[
\frac1{\mi} \psi_t(t,x) -\Delta \psi(t,x) = f(t,x_0).
\]
Due to the kernel  being non-local and highly oscillatory, accurate and efficient discretization of \eqref{eq:Sconv} is not easy. In \cite{Negulescu} a carefully constructed, accurate numerical method is presented and in \cite{Lub92} numerical experiments illustrated the good computational properties of convolution quadrature (CQ) for \eqref{eq:Sconv}.  Due to the non-locality of the kernel, computing $\psi(x,t)$ at $N$ time steps $t_j$ using the method in \cite{Negulescu} has an  $O(N^2)$ computational complexity, whereas standard FFT based methods for CQ, \cite{LU_88II,lb_ms} can reduce this to $O(N\log N)$ for linear problems and to $O(N\log^2N)$ for nonlinear problems \cite{HaLuSch,lb_ms}. However, both methods require to store $N$ solution vectors in memory. In this work we describe an algorithm that is both very easy to implement and can significantly reduce the amount of memory used. More precisely the memory requirements will be reduced from $O(N)$ to $O(n_0+ \log N)$, with $n_0\ll N$. 
%Algorithms  able to compress the memory in this way are usually said to be {\em% oblivious} \cite{SchaLoLu,BanLoSch}. 
Our method belongs to the family of {\em oblivious} algorithms \cite{SchaLoLu,BanLoSch}, the name indicating  that the memory requirements can be significantly reduced. 
Due to the way our algorithm is built we also expect it to be extendable to a variable step implementation, something which by construction is difficult for FFT based  methods \cite{HaLuSch,lb_ms}.

%for the implementation of CQ that  reduces the memory requirements to $O(N^{1/2})$ while keeping the computational complexity at $O(N \log N)$ both for evaluating the convolution and solving related convolutional equations. \lb{In fact the memory requirements increase as $O(h^{-1/2}+\log T/h) = O(N^{1/2}T^{-1/2}+\log N)$ for decreasing time-step $h > 0$. Hence the algorithm is particularly efficient for  a large $T$ and moderate $h > 0$.}

Oblivious quadratures were first developed for parabolic problems \cite{LuSch,SchaLoLu}, where the transfer operator is {\em sectorial}, i.e., it admits a holomorphic extension to the complement of an acute sector in the left half of the complex plane and it grows at most algebraically as $|z|\to \infty$. The extension of such ideas to hyperbolic problems, where the transfer operator typically exhibits exponential growth as $\real z \to -\infty$,  is much more recent and has been first developed for the two-dimensional and the damped three dimensional wave equation \cite{BanLoSch}. The application to the Schr\"odinger equation has never been addressed to our knowledge. Furthermore in this paper we use a new approach which allows a substantial simplification of the implementation with respect to the algorithms in \cite{LuSch,SchaLoLu,BanLoSch} and in our experience even a marginal improvement in the compressibility and the computational times. In particular, we take the real inverse Laplace transform approach introduced in \cite{BanLo} for the fractional integral, which fits the more favourable sectorial framework. We thus generalize the ideas in \cite{BanLo} to the non-sectorial situation of the Schr\"odinger Green's kernel and propose a {\em mosaic-free}, fast and oblivious algorithm for the approximation of \eqref{eq:Sconv}. Here {\em mosaic} refers to the special partition of the integration domain $\tau \le t$ and organization of the book-keeping which is required to compress the memory by the algorithms in \cite{LuSch,SchaLoLu,BanLoSch}. The algorithm we propose here will only require the computation and storage of certain quantities at the beginning of the integration procedure and the update of a unique set of ordinary differential equations from one step to the next one, for the whole time interval. More details about the implementation are given in Section~\ref{sec:param}.

The outline of the paper is as follows. In Section~\ref{sec.rkcq} we give a brief introduction to Runge-Kutta based convolution quadrature,  in Section~\ref{sec:int} we describe a new (real) integral representation of the convolution quadrature weights associated to the transfer operator \eqref{transfer_schrodinger}. An efficient quadrature of this representation of the convolution weights is described in Section~\ref{sec:quad}. The use of this quadrature in an efficient algorithm for computing discrete convolutions is explained in Section~\ref{sec:param}. The new method is then illustrated by several substantial numerical experiments in Section~\ref{sec:numerics}. In particular we describe in detail an application to a nonlinear Schr\"odinger equation describing the suppression of quantum beating taken from \cite{Negulescu}. The codes used for these experiments are published at \cite{codes}.

\section{Runge-Kutta convolution quadrature}
\label{sec.rkcq}

In this section we briefly describe convolution quadrature (CQ) as applied to the evaluation of one sided convolutions
\begin{equation}
  \label{eq:conv}
u(t) = \int_0^tk(t-\tau) g(\tau)d\tau,
\end{equation}
where $k$ is a given kernel with Laplace transform $K(z) = \mathscr{L} k (z)$ and given data $g$. In the applications in this paper, $k$ will also be a function of $x$, but as this dependence plays no role in this section we supress it for now. A basic assumption for the application of CQ to \eqref{eq:conv} is that there exist $C>0$ and $\mu \in \bR$ such that
\begin{equation}\label{mu}
|K(z)| \le C|z|^{\mu}, \quad \mbox{for } \real z >0.
\end{equation}
In this paper we use CQ based on implicit
$A$-stable Runge-Kutta methods \cite{HaWII}. We employ standard
notation for an $s$-stage Runge-Kutta discretization based on the
Butcher tableau described by the matrix $\mathbf{A} =
(a_{ij})_{i,j=1}^s \in \IR^{s\times s}$ and the vectors $\bv =
(b_1,\ldots,b_s)^T \in \IR^s$ and $\cv = (c_1,\ldots,c_s)^T \in
[0,1]^s$. The corresponding stability function is given by
\begin{equation}\label{stabfun}
r(z) = 1+z\bv^T (\mathbf{I}-z \mathbf{A})^{-1} \bone,
\end{equation}
where
\[
\bone = (1,1,\dots,1)^T.
\]
Recall that $A$-stability is equivalent to the condition $|r(z)| \leq 1$
for $\real z \leq 0$.  In the following we collect all the assumptions
on the Runge-Kutta method. These are satisfied by, for example, Radau
IIA and Lobatto IIIC families of Runge-Kutta methods.

\begin{assumption}\label{as:RK}
\begin{enumerate}[(a)]
%\item[(a)]
%The stability function can also be
%expressed as $R(z) = P(z)/Q(z)$ with
%polynomials $P(z) =
%\on{det}(\id-zA+z\bone b^T)$ and $Q(z) = \on{det}(\id-zA)$, see
%\cite{MR1439506}.  Let this representation be irreducible,
%i.e., let  $P(z)$ and $Q(z)$ be relatively prime.
%
\item
\label{as:RK-a}
The Runge-Kutta method is $A$-stable with (classical) order $p\ge 1$ and
stage order $q\leq p$.
\item
\label{as:RK-b}
The stability function satisfies
$|r(\mi y)| < 1$ for all real $y\ne 0$.
\item
\label{as:RK-c} The Runge-Kutta coefficient matrix
$\mathbf A$ is invertible.
\item
\label{as:RK-d}
The Runge-Kutta method is stiffly accurate, i.e.,
\[
\mathbf{b}^T\mathbf{A}^{-1}  =  (0, 0, \dots, 1).
\]
This implies that
\[
%r(\infty) =
\lim_{|z| \rightarrow \infty}r(z) = 1-b^TA^{-1}\bone  = 0
\]
 and $c_s = 1$.
\end{enumerate}
\end{assumption}

Since $r(z)$ is a rational function, the above assumptions imply that
\begin{equation}
  \label{eq:r_decay}
  r(z) = O(z^{-1}), \qquad |z| \rightarrow \infty.
\end{equation}

Following the theory in \cite{LubOst}, we define the weight matrices $\Wv_n$ corresponding to the operator $K$ as the coefficients of the power expansion
\begin{equation}\label{Wn}
  \sum_{n=0}^\infty \Wv_n\z^n = K\left( {\boldsymbol{\Delta}(\zeta) \over h}\right),
\end{equation}
where $h$ is the step size and the matrix-valued function $\boldsymbol{\Delta}(\zeta)$ is the so-called {\em symbol} of the Runge--Kutta method:
\begin{equation}\label{Delta}
  \boldsymbol{\Delta}(\zeta) = \Bigl(\mathbf{A} + {\z \over 1-\z}\bone \bv^T\Bigr)^{-1}.
\end{equation}
Denoting by $\omegav_n=(\omega_n^1,\dots,\omega_n^s)$ the last row of $\Wv_n$, the
approximation to the convolution integral \eqref{eq:conv} at time
$t_{n+1}=(n+1)h$ is given by
\begin{equation}\label{rk-cq}
 u_{n+1}=   \sum_{j=0}^n \sum_{i=1}^s  \omega_{n-j}^i\, g( t_j + c_ih) =
   \sum_{j=0}^n  \omegav_{n-j}\, \gv_j,
\end{equation}
with the column vector $\gv_j = g(t_j+\mathbf{c}h)=\bigl( g(t_j+c_ih) \bigr)_{i=1}^s$.

%For a Runge-Kutta method of classical order $p$ and stage order $q$,
The convergence order of this approximation has been investigated in
\cite{LubOst} for parabolic problems, i.e., for sectorial
$K$, and in \cite{BanL} and \cite{BanLM} for hyperbolic
problems, i.e., for non-sectorial operators.

With the row vector $\ev_n(z)=(e_n^1(z),\dots,e_n^s(z))$ defined as the last row
of the $s\times s$ matrix $\Ev_n(z)$ given by
\begin{equation}\label{En-def}
  (\Delta(\z)-zI)^{-1} = \sum_{n=0}^\infty \Ev_n(z) \,\z^n ,
\end{equation}
we obtain an integral formula for the weights
\begin{equation}\label{omegaRK-int}
  \omegav_n = {h\over 2\pi \mi}
  \int_\Gamma  K(z)\ev_n(hz)\, dz.
\end{equation}
This representation follows from Cauchy's formula and the definition of the weights in \eqref{Wn}, with the integration contour $\Gamma$ chosen so that it
surrounds the poles of $\ev_n(hz)$. An explicit expressing for $\ev_n$ is  given by
\begin{equation}\label{en-rk}
\ev_n(z) = r(z)^n \mathbf{q}(z),
\end{equation}
with the row vector $\mathbf{q}(z) =  \bv^T (I-z\mathbf{A})^{-1}$; cf.~\cite[Lemma~2.4]{LubOst}. The $A$-stability assumption implies that the poles of $r(z)$ are all in the right-half plane. Further, due to the decay of the rational function $r(z)$, see \eqref{eq:r_decay}, for $n > \mu+1$ with $\mu$ in \eqref{mu}, the contour $\Gamma$ can be deformed into the imaginary axis.

For the weight matrices it holds
\begin{equation}\label{Wmatrix-int}
  \Wv_n = {h\over 2\pi \mi}
  \int_\Gamma   K(z) \Ev_n(hz)\, dz.
\end{equation}
By \cite[Lemma~2.4]{LubOst},  for $n\ge 1$, $\Ev_n(z)$ is the rank-1 matrix given
by
\begin{equation}\label{En-formula}
\Ev_n(z) = r(z)^{n-1} (I-z\mathbf{A})^{-1} \bone \bv^T (I-z\mathbf{A})^{-1}.
\end{equation}
The Runge-Kutta approximation of the inhomogeneous linear problem
\begin{equation}
  \label{eq:ode}
  y'(t) = z y(t) + g(t), \quad y(0) = 0,
\end{equation}
at time $t_{n+1}$ is given by
\begin{equation}
  \label{eq:rkode}
  y_{n+1}(z) = h \sum_{j=0}^n \ev_{n-j}(hz) \gv_j
\end{equation}
and thus the approximation of the convolution integral in~\eqref{rk-cq} can be rewritten as~\cite[Proposition~2.4]{LubOst}
\begin{equation}
  \label{eq:ci-int}
  u_{n+1} = \frac{1}{2\pi\mi} \int_\Gamma K(z) y_{n+1}(z) dz.
\end{equation}

We will require the following technical lemmas proved in \cite{BanLoSch} and \cite{BanLo} where examples of numerically computed  values of $\gamma$ can also be found.

\begin{lemma}\label{lemma:gamma}
Let $r(z)$ be the stability function of a Runge-Kutta method satisfying Assumption~\ref{as:RK} and let
\[
\gamma(\xi) = \inf_{-\xi \leq \Re z \leq 0} \frac{\log |r(z)|}{\Re z}.
\]

Then $\gamma(\xi) \in (0,1]$ for  $\xi > 0$, it monotonically increases as $\xi \to 0$ and
\[
|r(z)| \leq e^{\gamma(\xi)\Re z},
\]
for all $z$ in the strip  $-\xi \leq \Re z \leq 0$.
\end{lemma}

\begin{lemma}\label{lem:rbound}
There exist constants $\nu > 1$, $b > 0$ and  $C_{\qv} > 0$ such that
\[
|r(z)| \leq e^{\nu \Re z}, \qquad \text{for } 0 \leq \Re z \leq b,
\]
and
\[
\|{\bf q}(z)\| \leq C_{\qv}, \qquad \text{for } \Re z \leq b,
 \]
where $C_{\qv}$ depends on the choice of the norm $\|\cdot \|$.
\end{lemma}
\begin{lemma}\label{lem:qbound}
There exists a constant $C_A > 0$ such that
\[
\max(r(z), \|{\bf q}(z)\|) \leq C_{A} |z|^{-1} \qquad \forall \Re z \leq 0.
  \]
\end{lemma}
\begin{proof}
  The estimate follows from  $r(z) = \bv^T\Av^{-1}(I-z\Av)^{-1}\bone$, $\mathbf{q}(z) = \bv^T(I-z\Av)^{-1}$, and the fact that the eigenvalues of $\Av$ have a strictly positive real part.
\end{proof}
\begin{remark}
For the backward Euler method, $C_A = 1$. For other RK methods the constant can be estimated numerically. We obtained that for the 2-stage Radau IIA method $C_A \approx 2.1213$, for the 3-stage Radau IIA $C_A \approx 3.479$, and for the 3-stage Lobatto IIIA $C_A \approx 3.6224$.
\end{remark}

In the next section we discuss how to approximate the integral
in~\eqref{omegaRK-int} by an efficient quadrature rule.

\section{Integral representation of the convolution quadrature weights}\label{sec:int}
We will follow the same idea as in the derivation of the real inversion formula for the Laplace transform in \cite[Section 10.7]{Henrici_II}, but with $\ev_n(hz)$ in place of $e^{zt}$. For the rest of the paper we will write
\begin{equation}
  \label{eq:d_defn}
  d = \|x\|.
\end{equation}
In the following results we will need some properties of the modified Bessel function $K_0$. First of all, by $K_0$ we mean the principal branch analytic in the cut complex plane $\mathbb{C} \setminus (-\infty, 0]$ as described in \cite[\S10.25]{NIST}. The  large argument behaviour is
    \begin{equation}
      \label{eq:K0_inf}
      \lim_{|z| \rightarrow \infty} \sqrt{z}e^z K_0(z) = \sqrt{\pi/2} \qquad |\arg(z)| \leq \pi;
    \end{equation}
see \cite[10.25.3]{NIST}. Whereas the small argument behaviour is
    \begin{equation}
      \label{eq:K0_log}
      \lim_{|z| \rightarrow 0}-K_0(z)/\log z = 1;
    \end{equation}
see \cite[10.25.3]{NIST}. From these two results it follows that there exists a constant $C_{K_0} > 0$ such that
\begin{equation}
  \label{eq:K0_bound}
  |K_0(z)| \leq C_{K_0} e^{-\Re z}|z|^{-1/2} \qquad |\arg(z)| < \pi.
\end{equation}

\begin{lemma}
For $n \geq 1$,
  \[
\omegav_n(x) = \frac{h}{2\pi} \int_{-\infty}^\infty \ev_n(-\mi hy) K(-\mi y,x)dy.
\]
\end{lemma}
\begin{proof}
  Note that for $\Re z \geq 0$, $\Re \sqrt{z}e^{-\mi\pi/4} \geq 0$ and hence $|K(z,x)| \leq \frac12 |z|^{-1/2}$ for $D = 1$, $|K(z,x)| \leq C_{K_0}|z|^{-1/2}$ for $D = 2$ and $|K(z,x)| \leq \frac1{4\pi d}$ for $D = 3$. As $\|\ev_n(z)\| \leq C_A^{n+1}|z|^{-n-1}$ we can deform the contour as required by the statement of the lemma.
\end{proof}

\begin{lemma}\label{lemma:kernel_bound}
  Given $d>0$ and $\xi > 0$
$$
f_\pm(y) = e^{-d\Re\left(e^{-\mi\pi/4}\sqrt{-\xi\pm\mi y}\right)}
$$
are decreasing functions of $y \geq 0$.
\end{lemma}
\begin{proof}
Note that it is sufficient to show that $g_{\pm}(y)$ decreases with $g_{\pm}$ the real part of the exponent:
\[
 g_{\pm}(y) := -d\Re (e^{-\mi \pi/4} \sqrt{-\xi\pm\mi y}).
\]
 Let us first consider $g_{-}(y)$:
\[
 g_{-}(y) = -d|\xi+\mi y|^{1/2}\cos(\arg(-\xi-\mi y)/2-\pi/4).
\]
Notice that the cosine above is negative. Then by using the half-angle formula for the cosine we obtain
\[
\begin{split}
g_{-}(y)&= d|\xi+\mi y|^{1/2} \sqrt{\frac{1+\cos(\arg(-\xi-\mi y)-\pi/2)}{2}}\\
& = d|\xi+\mi y|^{1/2} \sqrt{\frac{1-\sin(\arg(-\xi-\mi y))}{2}}\\
& =  d(\xi^2 + y^2)^{1/4} \sqrt{\frac{1-\frac{y/\xi}{\sqrt{1+y^2/\xi^2}}}{2}}
 = \frac{d}{\sqrt{2}} \sqrt{(\xi^2 + y^2)^{1/2}-y}.
\end{split}
\]
As $g'_{-}(y) = -\frac{d}{2\sqrt{2}}\sqrt{(\xi^2+y^2)^{1/2}-y}/\sqrt{\xi^2+y^2} \leq 0$, the function $f_{-}(y)$ decreases as $y$ increases.

For the exponent of $f_{+}(y)$ we obtain
\[
  \begin{split}
g_{+}(y) & = -d|-\xi+\mi y|^{1/2}\cos(\arg(-\xi+\mi y)/2-\pi/4)\\
& = -d|\xi+\mi y|^{1/2} \sqrt{\frac{1-\sin(\arctan(-y/\xi))}{2}}\\
& = -d(\xi^2 + y^2)^{1/4}\sqrt{\frac{1+\frac{y/\xi}{\sqrt{1+y^2/\xi^2}}}{2}}
% =-d (\xi^2 + y^2)^{1/4} \sqrt{\frac{1+\frac{y}{\sqrt{\xi^2+y^2}}}{2}}
\\& =-\frac{d}{\sqrt{2}} \sqrt{(\xi^2 + y^2)^{1/2}+y}.
\end{split}
\]
Clearly this function is decreasing, hence again $f_{+}(y)$ decreases as $y$ increases.
\end{proof}

\begin{theorem}\label{thm:contour}
The weights are given by
\begin{equation}\label{omega-newint}
  \begin{split}
  \omegav_n(x) = &{h\over 2\pi \mi}
    \int_{\Gamma_+ \cup \Gamma_-}  K(z,x)\ev_n(hz)\, dz \\ &+ {h\over 2\pi \mi}\int_0^{\xi} \left( K(\lambda e^{-i\pi},x)- K(\lambda e^{i\pi},x)\right)\ev_n(-h\lambda)\,d\lambda,
    \end{split}
\end{equation}
where
\begin{equation}\label{gammas1}
\Gamma_{\pm }=\{-\xi\pm \mi y: y>0\}.
\end{equation}
Further, the function  $G(\lambda,x)  := K(\lambda e^{-i\pi},x)- K(\lambda e^{i\pi},x)$ is given by
\[
G(\lambda,x) =
\left\{
  \begin{array}{cc}
    \displaystyle\frac{\mi e^{\mi \pi/4}}{\sqrt{\lambda}}\cosh(|x| e^{i\pi/4}\sqrt{\lambda})& D = 1,\\[1em]
        \displaystyle\frac{\mi}{2}J_0(  e^{-i\pi/4}\sqrt{\lambda}\|x\|)& D = 2,\\[1em]
    \displaystyle\frac{1}{2\pi \|x\|}\sinh(\|x\| e^{i\pi/4}\sqrt{\lambda})& D = 3.
  \end{array}
\right.
\]
\end{theorem}
\begin{proof}
We first deform the integration contour from the imaginary axis to the contour described in Figure~\ref{fig:contour}. Letting $R \rightarrow \infty$ and $\delta \rightarrow 0$ and using the estimate for $K(z)$ from the previous lemma for $D = 1$ and $D = 2$ and the bound on $e_n(z)$ as before we obtain the required expression \eqref{omega-newint}. As \eqref{eq:K0_bound} shows that the kernel in $D = 2$ dimensions is bounded up to a constant by the same expression as in $D = 1$ dimension, the same argument works here too.

  To derive the simplified expression for $G(\lambda,x)$ for the 3D Schr\"odinger equation note that
\begin{eqnarray*}
K(\lambda e^{-i\pi}, x)- K(\lambda e^{i\pi}, x)&=& \frac1{4\pi d}\left[\exp(-d e^{-i3\pi/4}\sqrt{\lambda})- \exp(-d e^{i\pi/4}\sqrt{\lambda})\right]\\
&=& \frac1{4\pi d}\left[\exp(d e^{i\pi/4}\sqrt{\lambda})- \exp(-d e^{i\pi/4}\sqrt{\lambda})\right] \\
&=& \frac1{2\pi d}\sinh(d e^{i\pi/4}\sqrt{\lambda}).
\end{eqnarray*}
Similarly in 1D we have
\begin{eqnarray*}
K(\lambda e^{-i\pi}, x)- K(\lambda e^{i\pi}, x)&=& \frac{e^{\mi \pi/4}}{2\sqrt{\lambda}e^{-\mi\pi/2}}\exp(-d e^{-i3\pi/4}\sqrt{\lambda})- \frac{e^{\mi \pi/4}}{2\sqrt{\lambda}e^{\mi\pi/2}}\exp(-d e^{i\pi/4}\sqrt{\lambda})\\
&=& \frac{\mi e^{\mi \pi/4}}{2\sqrt{\lambda}}\left[\exp(d e^{i\pi/4}\sqrt{\lambda})+ \exp(-d e^{i\pi/4}\sqrt{\lambda})\right] \\
&=& \frac{\mi e^{\mi \pi/4}}{\sqrt{\lambda}}\cosh(d e^{i\pi/4}\sqrt{\lambda}).
\end{eqnarray*}

In 2D we have
\begin{align*}
  K(\lambda e^{-i\pi}, x)- K(\lambda e^{i\pi}, x)&=
\frac1{2\pi}\left[K_0(e^{-\mi \pi/2}d\sqrt{\lambda} e^{-i\pi/4})- K_0(e^{\mi \pi/2}d\sqrt{\lambda} e^{-i\pi/4})\right]\\
&= \frac{\mi}{2} J_0(d\sqrt{\lambda} e^{-i\pi/4}),
\end{align*}
where in the last step we used \cite[10.27.9]{NIST}.
\end{proof}

  \begin{figure}
    \centering
    \includegraphics[width=.8\textwidth]{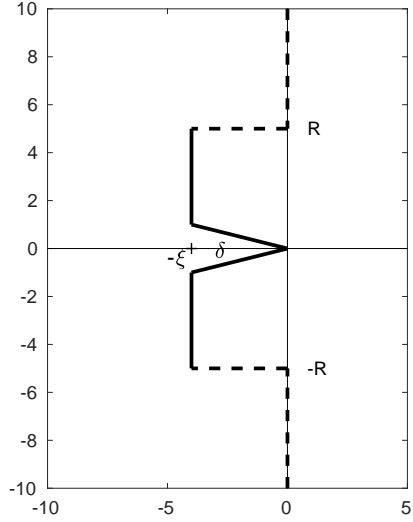}
    \caption{The contour in the proof of Theorem~\ref{thm:contour} is the  union of the dashed and solid lines. The angle at the negative real axis is denoted by $\delta > 0$ and $-\xi$ is the real part of the solid vertical lines.}
    \label{fig:contour}
  \end{figure}

\section{Quadrature for the convolution weights}\label{sec:quad}

In this section we develop an efficient quadrature for the approximation of the convolution quadrature weights, written as \eqref{omega-newint}. Let  $T$ be the final time, $h=T/N$ the time-step  and $N$ the total number of time steps in our approximation (see \eqref{rk-cq}).  Our goal is to use the same quadrature weights and nodes for the approximation of $\omegav_n(x)$ for all $n_0 < n \le N$, for some $n_0 \ll N$ to be determined, and all $0\le x \le L$, for $L$ the maximal distance between the points $x_j$ in \eqref{conschrodinger}.

In the first place we will bound the size of the contribution along the vertical lines $\Gamma_{\pm}$, showing that in some cases it can be neglected. Thus we will approximate
\begin{equation}\label{int0xi}
\omegav_n(x) \approx \Iv_n:= \frac{h}{2\pi\opi} \int_0^{\xi} G(\lambda,x)\ev_n(-h\lambda)\,d\lambda.
\end{equation}

For the rest of the paper we restrict ourselves to $D = 1$. In view of  the very similar expression for the 3D kernel and the bound \eqref{eq:K0_bound} in 2D, we expect that similar estimates will hold also in higher dimensions. The single bigger change would be the treatment of the singularity at the origin.

\subsection{Truncation to a finite interval}
We start by bounding the contribution to the integral along the vertical semilines, i.e. the error in \eqref{int0xi}.

\begin{proposition}\label{prop:truncation}
  We have the bound
 \begin{equation}
   \label{eq:truncation_en}
   \begin{split}
     \left\|  \omegav_n(x)-\Iv_n \right\| \leq &  \frac{h}{4\pi} \int_0^\infty \|\ev_n(h(\xi+\mi y))\| e^{-d \Re(e^{-\frac{\mi\pi}{4}}\sqrt{-\xi+\mi y})}(\xi^2+y^2)^{-\frac14}dy\\
     &+  \frac{h}{4\pi} \int_0^\infty \|\ev_n(h(\xi-\mi y))\| e^{-d \Re(e^{-\frac{\mi\pi}{4}}\sqrt{-\xi-\mi y})}(\xi^2+y^2)^{-\frac14}dy.
        \end{split}
      \end{equation}
or more explicitly
 \begin{equation}
   \label{eq:truncation}
   \left\|  \omegav_n(x)-\Iv_n \right\| \leq   C_A\frac{\sqrt{2\pi/\xi}}{2\Gamma(3/4)^2}\cosh(d\sqrt{\xi/2})e^{-\gamma(h\xi)t_n\xi}.
   \end{equation}
\end{proposition}
\begin{proof}
To prove the result we need to bound the integrals over $\Gamma_+ \cup \Gamma_-$ in Theorem~\ref{thm:contour}. The estimate \eqref{eq:truncation_en} follows directly from the definition of the integrand.

 Note that Lemma~\ref{lemma:kernel_bound} implies
\[
|K(-\xi-\mi y,x)| \leq \frac12 (\xi^2+y^2)^{-1/4}e^{d\sqrt{\xi/2}}
\]
and
\[
|K(-\xi+\mi y,x)| \leq    \frac12 (\xi^2+y^2)^{-1/4}e^{-d\sqrt{\xi/2}}.
\]
We require a bound on $\ev_n(h(-\xi\pm \mi y))$ which follows from $\ev_n(z) = r(z)^n \qv(z)$, Lemma~\ref{lemma:gamma} and Lemma~\ref{lem:qbound}
\[
\|\ev_n(h(-\xi\pm \mi y))\|  \leq C_Ae^{-\gamma(h\xi)t_n\xi}\frac{h^{-1}}{(\xi^2+y^2)^{1/2}}.
\]

Hence
\[
  \begin{split}
  \left|\frac{h}{2\pi \mi}\int_{\Gamma_-}  K(z,x)\ev_n(hz)\, dz\right|
  &\leq  C_A\frac{1}{4\pi}e^{d\sqrt{\xi/2}-\gamma(h\xi)t_n\xi} \int_0^\infty(\xi^2+y^2)^{-3/4} dy
  \\
  &= C_A\frac{1}{4\pi\xi^{1/2}}e^{d\sqrt{\xi/2}-\gamma(h\xi)t_n\xi} \int_0^\infty(1+y^2)^{-3/4} dy\\
  &= C_A\frac{\sqrt{2\pi/\xi}}{4\Gamma(3/4)^2}e^{d\sqrt{\xi/2}-\gamma(h\xi)t_n\xi}.
    \end{split}
  \]
Similarly
 \[
  \left|\frac{h}{2\pi \mi}\int_{\Gamma_+}  K(z,x)\ev_n(hz)\, dz\right|
  \leq  C_A\frac{\sqrt{2\pi/\xi}}{4\Gamma(3/4)^2}e^{-d\sqrt{\xi/2}-\gamma(h\xi)t_n\xi}.
\]
\end{proof}

The following corollary will be used to determine the value of $n_0$. More details are given in Section~\ref{sec:param}.
\begin{corollary}\label{coro:trunc}
  For a given $\varepsilon > 0$ and $\xi > 1$
\[
\left|  \omegav_n(x)-\Iv_n \right| \leq \varepsilon
\]
if
\[
n \geq \frac1{h\xi \gamma(h\xi)}\left(d \sqrt{\xi/2}+\log\left(\frac{C_A\sqrt{2\pi}}{4\sqrt{\xi}\Gamma(3/4)^2 \varepsilon}\right)\right).
\]
\end{corollary}

In order to efficiently approximate $\Iv_n$, we use the splitting
\begin{equation}\label{splitI}
\Iv_n =\sum_{j=0}^{J} \Iv_{n,j},
\end{equation}
with
\begin{equation}\label{Inj}
\Iv_{n,j} := \frac{h}{2\pi\opi} \int_{L_{j-1}}^{L_j} G(\lambda,x) \ev_n(-h\lambda)\, d\lambda,
\end{equation}
where $L_{-1} = 0$, $L_0 > 0$ is a free parameter, and $L_j=(1+B)L_{j-1}$ for $j > 1$ and some fixed $B\ge 1$. Every sub-integral $\Iv_{n,j}$ will be approximated by an appropriate Gauss quadrature. The case of $\Iv_{n,0}$ is treated separately due to the integrable singularity of the integrand at 0. To analyse the error due to Gauss quadrature, we use the following classical result.

\begin{theorem}\label{th:gwerr}
  Let $f$ be analytic inside the Bernstein ellipse
\[
\mathcal{E}_{\vrho} = \{ z \; : \; z = \frac12(w+w^{-1}), \, |w| = \vrho\}
\]
with $\vrho > 1$ and bounded there by $M$. Then the error of Gauss quadrature with weight $w(x)$ is bounded by
\[
|If-I_Q f| \leq 4M\frac{\vrho^{-2Q+1}}{\vrho-1}\int_{-1}^1 w(x)dx,
\]
where $If = \int_{-1}^1 w(x) f(x) dx$ and $I_Q f = \sum_{j = 1}^Q w_j f(x_j)$ is the corresponding Gauss formula, with weights $w_j > 0$.
\end{theorem}
\begin{proof}
 To the best of our knowledge, the first proof of this result appeared in \cite{Sy77}. A different proof for $w(x) \equiv 1$, which can easily be extended to the case of a general weight as in \cite{BanLo},  can be found in \cite[Chapter 19]{Tre}.
\end{proof}

\subsection{Gauss-Jacobi quadrature for the initial interval}
We fix the first interval $[0,L_0]$ and compute
\[
\Iv_{n,0} = \frac{h}{2\pi\mi}\int_0^{L_0} G(\lambda,x) \ev_n(-h\lambda)\,d\lambda
= \frac{hL_0}{4\pi \mi} \int_{-1}^1 G((1+y)L_0/2,x) \ev_n(-h(1+y)L_0/2) \,dy.
\]
In the 1D case, $G(\lambda,x)\sqrt{\lambda} = \mi e^{\mi \pi/4}\cosh(d e^{\mi\pi/4}\sqrt{\lambda})$ is an entire function of $\lambda$. Hence in this interval we will use Gauss-Jacobi quadrature with weight $w(x) = (x+1)^{-1/2}$ on the interval $[-1,1]$. We denote by $\tauv_{n,0}(Q)$ the corresponding quadrature error when taking $Q$ quadrature nodes.

% Recall the Bernstein ellipse $\mathcal{E}_{\vrho}$, which is given as the image of the circle of radius $\varrho > 1$ under the map $z \mapsto (z+z^{-1})/2$. The largest imaginary part on $\mathcal{E}_{\rho}$ is $(\varrho-\varrho^{-1})/2$ and the largest real part is $(\varrho+\varrho^{-1})/2$.
To estimate the error of the quadrature, we will need to bound, according to Theorem \ref{th:gwerr},
\begin{equation}\label{fv0}
\fv(\zeta) =  h\frac{\sqrt{2L_0}}{4\pi} \ev_n(-h(1+\zeta)L_0/2)\cosh(d e^{i\pi/4}\sqrt{L_0(1+\zeta)/2}), \qquad \zeta \in \mathcal{E}_{\vrho},
\end{equation}
where we have already neglected the modulus one quantity $\e^{\opi\pi/4}$ included in the definition of $G(\lambda,d)$. Notice that there is a maximal value of $\vrho$, which we denote $\vrho_{\max}$, determined by the location of the poles of $\ev_n$.

\begin{theorem}\label{th:errgj}
Let $b$ and $\nu$ be as in Lemma~\ref{lem:rbound},
\begin{equation}\label{rhomax_gj}
\vrho_{\max}=1+\frac{2b}{L_0 h}+\sqrt{\left(\frac{2b}{L_0 h}\right)^2+\frac{4b}{L_0h}}
\end{equation}
and
\begin{equation}\label{rhoopt_gj}
\vrho_{opt}=\frac{2Q}{ \left(d\sqrt{3L_0/2}+ \nu t_n  L_0/2\right)} + \sqrt{1+\left(\frac{2Q}{ \left(d\sqrt{3L_0/2}+ \nu t_n L_0/2\right)}\right)^2}.
\end{equation}
Then, if $\vrho_{opt} \in (2+\sqrt{3},\vrho_{\max})$, we can bound the error of the $Q$-node Gauss-Jacobi quadrature by
\[
\left\| \tauv_{n,0}(Q) \right\| \leq C_{\qv} h\frac{4\sqrt{L_0}}{\pi} \frac{\vrho_{opt}}{\vrho_{opt}-1} \left(\frac{ \e \left(d\sqrt{3L_0/2}+ \nu  t_n L_0/2\right)}{4Q}\right)^{2Q}.
\]
Otherwise
\[
\|\tauv_{n,0}(Q) \| \leq C_{\qv} h\frac{4\sqrt{L_0}}{\pi} \frac{\vrho_{\max}^{-2Q+1}}{\vrho_{\max}-1} \exp\left( \frac{ b}{h}\left(d\sqrt{6/L_0}+ \nu t_n \right)\right).
\]
\end{theorem}
% \red{This seems to suggest taking $L_0 \sim \min(1/T,1/d^2)$.}
\begin{remark}\label{rem:GJQ}
The second estimate, seems quite pessimistic because of the exponentially growing term. However, notice that we can take $b = h$  which implies $\nu \approx 1$ and $\vrho_{\max} > 1+4/L_0$. Hence both estimates imply that if we choose $L_0 \propto T^{-1}$, $Q = O(\log \veps)$ quadrature nodes are sufficient to obtain $\|\tauv_{n,0}(Q) \| \leq \veps$.
Note, that in numerical experiments reported in this paper, it was always the case that $\vrho_{opt} \in (2+\sqrt{3},\vrho_{\max})$.
\end{remark}
\begin{proof}[Theorem~\ref{th:errgj}]
Note that
\[
|\cosh(d e^{i\pi/4}\sqrt{L_0(1+\zeta)/2})| \leq e^{d \left|\Re e^{i\pi/4}\sqrt{L_0(1+\zeta)/2} \right|} \leq e^{d|L_0(1+\zeta)/2|^{1/2}}.
\]
As we will want to avoid the poles of $\ev_n$ and use Lemma~\ref{lem:rbound} we need that $\vrho \leq \vrho_{\max}$. This upper bound \eqref{rhomax_gj} is obtained as a solution of
\[
h\left(\vrho_{\max}+\vrho^{-1}_{\max}-2\right)L_0/4=b.
\]
From Lemma~\ref{lem:rbound} and the definition of $\ev_n$ we can now bound
\[
  \begin{split}
\|\fv(\zeta)\| &\leq C_{\qv} h\frac{\sqrt{2L_0}}{4\pi} \exp\left({d|L_0(1+\zeta)/2|^{1/2}-\nu t_n L_0(1+\Re\zeta)/2}\right), \qquad \Re \zeta \leq -1, \zeta \in \mathcal{E}_{\vrho}\\
&\leq  C_{\qv} h\frac{\sqrt{2L_0}}{4\pi} \exp\left(dL_0^{1/2}(\vrho+\vrho^{-1}-2)^{1/2}/2+\nu t_n L_0(\vrho+\vrho^{-1}-2)/4\right).
  \end{split}
\]
Let $\vrho = e^{\delta}$ for $\delta > 0$. Then
\[
  \begin{split}
\|\fv(\zeta)\| &\leq  C_{\qv} h\frac{\sqrt{2L_0}}{4\pi} \exp\left(dL_0^{1/2}(\cosh \delta -1)^{1/2}/\sqrt{2}+\nu t_n L_0(\cosh \delta -1)/2\right).
\end{split}
\]
Assuming $\vrho  \geq 2+\sqrt{3}$ implies $\cosh \delta \geq  2$ and
\[
  \begin{split}
\|\fv(\zeta)\| &\leq  C_{\qv} h\frac{\sqrt{2L_0}}{4\pi}  \exp\left(\left(d\sqrt{L_0/2}+ \nu t_n L_0/2 \right)(\cosh \delta -1)\right), \qquad \Re \zeta \leq -1, \zeta \in \mathcal{E}_{\vrho}.
\end{split}
\]

For $\Re \zeta \geq -1, \zeta \in \mathcal{E}_{\vrho}$ we have that $|r(-h(1+\zeta)L_0/2)| \leq 1$ and hence
\[
  \begin{split}
\|\fv(\zeta)\| &\leq C_{\qv} h\frac{\sqrt{2L_0}}{4\pi} e^{d|L_0(1+\zeta)/2|^{1/2}}\\
&\leq C_{\qv} h\frac{\sqrt{2L_0}}{4\pi} e^{d\sqrt{L_0/2}|1+\cosh\delta|^{1/2}}\\
&\leq C_{\qv} h\frac{\sqrt{2L_0}}{4\pi} e^{d\sqrt{3L_0/2}(\cosh \delta -1)},
  \end{split}
\]
where we have used that $(1+x)^{1/2} \leq \sqrt{3}(x-1)$ for $x \geq 2$.

Therefore, from Theorem \ref{th:gwerr} we deduce that
\[
\| \tauv_{n,0}(Q)\| \leq C_{\qv} h\frac{4\sqrt{L_0}}{\pi}\min_{\delta > 0} \frac{e^{\delta}}{e^{\delta}-1}\exp\left(-2Q\delta + \left(d\sqrt{3L_0/2} + \nu t_n L_0/2 \right)(\cosh \delta -1)\right).
\]
So we minimize
\[
g(\delta) = -2Q\delta+\left(d\sqrt{3L_0/2} + \nu t_n L_0/2 \right)(\cosh \delta -1).
\]
As
\[
g'(\delta) = -2Q+  \left(d\sqrt{3L_0/2} + \nu t_n L_0/2 \right)\sinh \delta
\]
and
\[
g''(\delta) = \left(d\sqrt{3L_0/2} + \nu t_n L_0/2 \right)\cosh \delta \geq 0
\]
the minimum is reached at
\[
\delta_{opt} = \sinh^{-1} \left(\frac{2Q}{d\sqrt{3L_0/2} + \nu t_n L_0/2 }\right).
\]
Using the identities
$$
\sinh^{-1}y = \log\left( y+\sqrt{1+y^2} \right), \quad \cosh x = \sqrt{1+\sinh^2 x},
$$
we obtain the value of $\vrho_{opt}=\e^{\delta_{opt}}$ in the statement and
\[
e^{-2Q\delta_{opt}} \leq \left(\frac{d\sqrt{3L_0/2} + \nu t_n L_0/2 }{4Q}\right)^{2Q}.
\]
Using now that $-1+\sqrt{1+x^2} \leq x$, for $x \geq 0$, we have
\[
\exp(\cosh \delta_{opt} -1) \leq \exp(\sinh \delta_{opt}) = \exp \left(\frac{2Q}{d\sqrt{3L_0/2} + \nu t_n L_0/2}\right).
\]
Hence, in case $\vrho_{opt} \in (1,\vrho_{\max})$ the following bound holds
\[
\|\tauv_{n,0}(Q)\| \leq  C_{\qv} h\frac{4\sqrt{L_0}}{\pi}  \frac{\vrho_{opt}}{\vrho_{opt}-1} \left( \e \frac{d\sqrt{3L_0/2}+ \nu  t_n L_0/2}{4Q}\right)^{2Q}.
\]
Otherwise we choose $\vrho=\vrho_{\max}=\e^{\delta_{\max}}$ and obtain
\[
\begin{split}
\|\fv(\zeta)\| &\le  C_{\qv} h\frac{\sqrt{2L_0}}{4\pi}  \exp\left( \left(d \sqrt{3L_0/2}+ \nu t_n L_0/2\right)(\cosh \delta_{\max} -1)\right) \\
&=  C_{\qv} h\frac{\sqrt{2L_0}}{4\pi}  \exp\left( \left(d \sqrt{6/L_0}+ \nu t_n \right) \frac{b}{h}\right)
\end{split}
\]
and the stated bound for the error.
\end{proof}

\subsection{Gauss quadrature away from the singularity}
In this section we analyze the error in the Gauss-Legendre ($w(x)=1$) quadrature of the integrals $\Iv_{n,j}$ in \eqref{splitI}, with $j\ge 1$, which can be written as
%\begin{equation}\label{Inj1}
\[
\Iv_{n,j}=h \frac{\Delta L_{j}}{4\pi\opi} \int_{-1}^{1} G\left(L_{j-1}+ \frac{\Delta L_j}{2} (y+1),x \right) \ev_n \left(-h \left(L_{j-1}+ \frac{\Delta L_j}{2} (y+1) \right) \right) \,dy,
%\end{equation}
\]
with $\Delta L_j:=L_j-L_{j-1}=BL_{j-1}$.

\begin{theorem}\label{th:errg}
Let $\tauv_{n,j}(Q)$ be the error in the approximation of $\Iv_{n,j}$ by Gauss quadrature with weight $w(x)=1$ and $Q$ quadrature nodes. Then
\[
\left\| \tauv_{n,j}(Q)\right\| \le  h\frac{2BL^{1/2}_{j-1}}{\pi}  \min_{\vrho \in (1,\vrho_{\max})} \frac{\vrho^{-2Q+1}}{\vrho-1}\max_{\theta \in [0,\pi]} h_{n,j}(\vrho,\theta),
\]
with
\begin{equation}
  \label{eq:heps}
h_{n,j}(\veps) = \eta_\theta^{-1/2}  \e^{dL_{j-1}^{1/2}(1+(\vrho/2)^2)^{1/4}\eta_\theta^{1/2} -\gamma(h L_{j-1}\eta_\theta)\eta_\theta  L_{j-1}t_n }
\end{equation}
and
\begin{equation}\label{eq:vrho_max}
\vrho_{\max}=1+\frac{2}{B}\left(1+ \sqrt{1+B} \right).
\end{equation}
and
\begin{equation}
  \label{eq:eta_theta}
\eta_\theta = (1+B((\vrho+\vrho^{-1})\cos \theta+2)/4).
\end{equation}
\end{theorem}
\begin{proof}
According to Theorem \ref{th:gwerr} we now need to bound the function
\begin{equation}\label{fvj}
\fv(\zeta) = h \frac{\Delta L_{j}}{4\pi\opi} G\left(L_{j-1}+ \frac{\Delta L_j}{2} (\zeta+1),x \right) \ev_n \left(-h \left(L_{j-1}+ \frac{\Delta L_j}{2} (\zeta+1) \right) \right) , \qquad \zeta \in \mathcal{E}_{\vrho},
\end{equation}
In order to avoid the singularity of the square root we require
$$
L_{j-1}-\frac{\Delta L_j}{4}(\vrho+\vrho^{-1}-2)
=L_{j-1}\left(1-\frac{B}{4}(\vrho+\vrho^{-1}-2)\right)
> 0,
$$
which is satisfied for $1<\vrho <\vrho_{\max}$ and
$$
\vrho_{\max}=1+\frac{2}{B}\left(1+ \sqrt{1+B} \right).
$$
Note  that (using $|z| \leq \sqrt{1+(\Im z/\Re z)^2}|\Re z|$)
\[
  \begin{split}
    \left|L_{j-1}+\frac{\Delta L_j}{2}(\zeta+1) \right| &\leq
    \sqrt{1+\frac{(\vrho-\vrho^{-1})(L_j-L_{j-1})/4}{(L_{j-1}+L_j)/2}} \left(L_{j-1}+\frac{\Delta L_j}{2}(\Re\zeta+1) \right)\\
        &=
        \sqrt{1+\frac{(\vrho-\vrho^{-1})^2(1+B)^2}{4(2+B)^2}}
        \left(L_{j-1}+\frac{\Delta L_j}{2}(\Re\zeta+1)\right) \\
        &\leq
        \sqrt{1+(\vrho/2)^2}
        \left(L_{j-1}+\frac{\Delta L_j}{2}(\Re\zeta+1) \right) \\
        &= \sqrt{1+(\vrho/2)^2}L_{j-1}
        \eta_\theta,
\end{split}
  \]
where
$\eta_\theta = 1+B((\vrho+\vrho^{-1})\cos \theta+2)/4$.

With this notation we can bound, for every  $\zeta \in \mathcal{E}_{\vrho}$, $\vrho \in (1, \vrho_{\max})$
\begin{align*}
\left\|\fv(\zeta)  \right\| &\le h \frac{\Delta L_{j}}{4\pi} \left| L_{j-1}+ \frac{\Delta L_j}{2} (\zeta+1) \right|^{-1/2} \e^{d \left| L_{j-1}+ \frac{\Delta L_j}{2} (\zeta+1) \right|^{1/2}} \left\|\ev_n \left(-h \left(L_{j-1}+ \frac{\Delta L_j}{2} (\zeta+1) \right) \right)\right\| \\
%&\le  h\frac{\Delta L_{j}}{4\pi} \xi_\zeta^{-1/2} \e^{d(1+(\vrho/2)^2)^{1/4}\xi_\zeta^{1/2} } e^{-\gamma(h\xi_\zeta)\xi_\zeta t_n }
  &\leq  h\frac{\Delta L_{j}}{4\pi} (L_{j-1}\eta_\theta)^{-1/2} \e^{dL_{j-1}^{1/2}(1+(\vrho/2)^2)^{1/4}\eta_\theta^{1/2} -\gamma(hL_{j-1}\eta_\theta)\eta_\theta L_{j-1}t_n }\\
 &=  h\frac{BL^{1/2}_{j-1}}{4\pi} \eta_\theta^{-1/2}  \e^{dL_{j-1}^{1/2}(1+(\vrho/2)^2)^{1/4}\eta_\theta^{1/2} -\gamma(hL_{j-1}\eta_\theta)\eta_\theta  L_{j-1}t_n }.
\end{align*}
 The result then follows from Theorem \ref{th:gwerr}.
\end{proof}

\begin{remark}
  The above result is somewhat unsatisfactory as it still contains a min-max problem. The following corollary simplifies the estimate but is too pessimistic in practice. Hence, we make use of the corollary only for the discussion about the complexity of the algorithm and in practice numerically solve the above min-max problem in order to obtain optimal parameters.
\end{remark}

\begin{corollary}\label{cor:errg}
With notation as in Theorem~\ref{th:errg}
\[
  \left\| \tauv_{n,j}(Q)\right\| \le  h\frac{2BL^{1/2}_{j-1}}{\pi}  \min_{\vrho \in (1,\vrho_{\max})} \frac{\vrho^{-2Q+1}}{\vrho-1}
  \eta_{-}^{-1/2}  \e^{dL_{j-1}^{1/2}(1+(\vrho/2)^2)^{1/4}\eta_{+}^{1/2} -L_{j-1}t_n\gamma(h\eta_{+})\eta_{-} }
\]
where
\[
  \eta_\pm = 1+B(\pm(\vrho+\vrho^{-1})+2)/4.
  \]
\end{corollary}
\begin{proof}
  The proof follows from the fact that $\gamma(\xi)$ decreases for increasing argument and that $\eta_\theta$ from Theorem~\ref{th:errg} decreases from $\theta = 0$ to $\theta = \pi$.
\end{proof}
\begin{remark}\label{rem:GQ}
  To understand the required number of quadrature points let us set $B = 3$, as in the numerical experiments, giving $\vrho_{\max} = 3$. Choosing $\vrho = 2$ we get $\eta_{+} = 35/8$ and $\eta_{-} = 5/8$ and
  \[
      \left\| \tauv_{n,j}(Q)\right\| \leq  ChL^{1/2}   4^{-Q}
      \e^{c_{+}dL_{j-1}^{1/2} -c_{-}L_{j-1}t_n}
      \leq ChL^{1/2}   4^{-Q}
      \e^{\frac{c_{+}^2 d^2}{4c_{-}t_n}},
\]
where $C = \frac{12}{\pi}\eta_{-}^{-1/2}$, $c_{+} = 2^{1/4}\eta_{+}^{1/2}$ and $c_{-} =  \gamma(hL_{j-1}\eta_{+})\eta_{-}$. Hence to obtain accuracy $\veps$ we need to choose $Q = O(\log \veps^{-1} + \frac{d^2}{t_n})$ quadrature nodes. Furthermore, this indicates that the integrand can get exponentially large as $t_n \rightarrow 0$ which can create difficulties in finite precision arithmetic.
\end{remark}
% \section{The damped wave equation}
% Let $\alpha>0$ be the damping parameter. Writing
% $$
% \sqrt{s^2+\alpha s}=\sqrt{s}\sqrt{s+\alpha}
% $$
% I obtain, for
% $$
% K(s)=e^{-\sqrt{s}\sqrt{s+\alpha}}
% $$
% that
% \begin{equation}
% G(\lambda,d)= \left\{ \begin{array}{ll}
% 2i \sin(\sqrt{\lambda(\alpha-\lambda)}), \quad & \mbox{if } \ 0<\lambda<\alpha, \\[.5em]
% e^{-\sqrt{\lambda(\lambda-\alpha)}}-e^{\sqrt{\lambda(\lambda-\alpha)}}, \quad & \mbox{if } \ \lambda>\alpha.
% \end{array}
% \right.
% \end{equation}
% This leads, for any $\xi>\alpha$, to the representation
% \begin{equation}\label{omega-damped}
% \begin{array}{l}
%   \omegav_n(d) \\ \displaystyle = {h\over \pi \mi}\left(
%   2\int_{\Gamma_{\pm 1}}  K(s,d)\ev_n(hs)\, ds + i \int_0^{\alpha} \sin\left(\sqrt{\lambda(\alpha-\lambda)}\right)\ev_n(-h\lambda)\,d\lambda -
%    \int_{\alpha}^{\xi} \sinh(\sqrt{\lambda(\lambda-\alpha)}) \ev_n(-h\lambda)\,d\lambda
%   \right).
%   \end{array}
% \end{equation}

\section{Parameter choice and the fast method}\label{sec:param}
In this section we specify the choice of the most relevant parameters in our method and explain our algorithm.

In the first place we choose the truncation parameter $\xi$. The estimate in Co\-ro\-lla\-ry~\ref{coro:trunc} suggests choosing $\xi \propto h^{-1}$. Hence, let $\xi = A_0/h$ for some constant $A_0 > 0$. Then if $n_0$ is chosen such that
\[
  n_0 \geq \frac{d \sqrt{\frac{A_0}{2h}}+\log \left(\frac{C_A \sqrt{2\pi h}}{4\sqrt{A_0}\Gamma(3/4)^2 \veps_0}\right)}{A_0\gamma(A_0)},
  \]
    the error due to truncation is less than $\veps_0$ for $n \geq n_0$. Note that this means that at least $O(h^{-1/2})$ weights will need to be computed directly. Here we use the estimate \eqref{eq:truncation}. Some minor gains can be made by computing the truncation error numerically using \eqref{eq:truncation_en}.

  Next, according to Theorem~\ref{th:errgj}  we need to choose $L_0$ proportional to $T^{-1}$, for $T$ the final time in our approximation as stated at the beginning of Section~\ref{sec:quad}. Hence we set $L_0  = A_1/T$ for some constant $A_1 > 0$. There will then be $J$ intervals away from the singularity where Gauss quadrature is used with $L_0 (1+B)^J = \xi$ implying $J = \frac1{(1+B)}\log\left(\frac{A_0T}{h}\right) = O(\log(T/h))$.

  Once all these parameters are set we can choose the number of quadrature nodes and weights in each interval so that each quadrature error $\tauv_{n,j}(Q) \leq \veps$ for $j = 0,\dots, J$ and $n \geq n_0$. Then with the choice $\veps_0 = \tol/2$ and $\veps =\frac{ \tol}{2(J+1)}$ we have that the total error is bounded by the tolerance $\tol>0$.   As explained in Remark~\ref{rem:GJQ}, $Q = O(\log \veps)$ guadrature nodes of the Gauss-Jacobi quadrature are sufficient to obtain error of size $\veps$ in this first interval. According to Remark~\ref{rem:GQ}, we need to use $O\left(\log \veps^{-1}+\frac{d^2}{n_0h}\right)$ quadrature nodes in each of the remaining intervals.

  Let us now describe the fast method for computing the discrete convolution
  $\uv_n = \sum_{j = 0}^n  \omegav_{j}(x)\fv_{n-j}$ for a general sequence of $s\times 1$ vectors $\fv_{j}$, $j = 0,\dots, N$. Using this formula requires the full sequences $\omegav_j$ and $\fv_j$, for $j=0,\dots,N$ to be kept in memory requiring $O(N)$ memory.
  This can instead be evaluated most efficiently by applying our quadrature approximation of the convolution weights in the same way as in \cite[Section 7.2]{BanLo}. We thus split the sum into two terms, the {\em local term}, with summation index $j=1,\dots,n_0$, and the remaining {\em history term}:
\begin{equation}
  \label{eq:model_dq}
\uv_n = \sum_{j = 0}^{n_0}  \omegav_{j}(x)\fv_{n-j} +\sum_{j = n_0+1}^n  \omegav_{j}(x)\fv_{n-j}.
\end{equation}
 The {\em local term} is evaluated directly, by precomputing and keeping in memory the first $n_0+1$ convolution weights.  The {\em history term} is instead evaluated by means of a fast summation algorithm which is based on the quadrature developed in Section~\ref{sec:quad}. After replacing the CQ weights by the result of applying our quadrature  we are led to the formula
$$
\sum_{j = n_0+1}^{n} \omegav_{j}(d)\fv_{n-j}
  \approx \dt\sum_{k = 1}^{N_Q} w_k G(x_k,d)  (r(-\dt x_k))^{n_0+1} \Qv_{n,k},
$$
with
$$
 \Qv_{\ell,k}=\sum_{j = 0}^{\ell-n_0-1}(r(-\dt x_k))^{j}  \qv(-\dt x_k)   \fv_{\ell-n_0-1-j}
$$
satisfying the recursion
$$
 \Qv_{\ell,k}= r(-\dt x_k)\Qv_{\ell-1,k}+  \qv(-\dt x_k) \fv_{\ell-n_0-1}, \quad \Qv_{n_0,k}=0.
$$

Let us investigate the complexity and the memory requirements of the above described algorithm. The vectors $\fv_j$, $j = n-n_0-1,n-n_0, \dots, n$ and $\Qv_{n-1,k}$, $k = 1,\dots, N_Q$ need to be kept in memory at each time step $t_n$, i.e., $n_0 + N_Q$ vectors need to be stored at any time. Recalling now that $N_Q = O\left(\left(\log \veps^{-1}+\frac{d^2}{n_0h}\right)\log(T/h)\right)$  the total memory requirements are given by  $O\left(n_0+\left(\log \veps^{-1}+\frac{d^2}{n_0h}\right)\log(T/h)\right)$. If we choose $n_0 = O(h^{-1/2})$, then the memory requirements are reduced to $O((h^{-1/2} +\log \veps^{-1})\log h^{-1})$. Turning to computational complexity, if the local term is computed using FFT methods as described in \cite{lb_ms,HaLuSch}, the total computational cost is  $O(n_0\log^2 n_0 + N N_Q)$. This reduces to $O(h^{-3/2}\log h^{-1})$ if we choose $n_0 = O(h^{-1/2})$.

We need to say a few more words about the choice of $n_0$. Any choice of $n_0$ as a function of $h$ that implies $t_{n_0} \rightarrow 0$ and $h \rightarrow 0$ would allow for exponentially increasing integrand, see Remark~\ref{rem:GQ}. In finite precision arithmetic this could lead to destructive cancellation and complete loss of accuracy. Choosing $n_0 = O(h^{-1})$ would entirely remove this difficulty however would require us to use many more direct steps and more memory. Nevertheless, even so our algorithm would bring many advantages if $n_0 \ll N$; compare this with \cite{BanLoSch}. Hence, $n_0$ should be chosen between $\text{\em const}\cdot  h^{-1/2}$ and $\text{\em const}\cdot  h^{-1}$ depending on the parameters of the problem investigated. In the two applications that we investigate in the numerical experiments, we choose the smallest possible $n_0$ that ensures that the truncation error, see Proposition~\ref{prop:truncation}, is bounded by the tolerance, i.e., $n_0 = O(h^{-1/2})$.
\section{Numerical experiments}\label{sec:numerics}

Next, we present results of various experiments to illustrate the new method. The codes used to perform these experiments can be found at \cite{codes}.

\subsection{Approximation of CQ weights}
In the first place we test our new quadrature and compare the CQ weights we obtain with those given by the standard method based on FFT.

We start by setting the parameters as in Section~\ref{sec:param} with $A_0 =1$ and $A_1 = 2$. We fix $B = 3$ and given a tolerance $\tol >0$ compute the number of nodes in each interval so that the error is bounded as described in Section~\ref{sec:param}. The results for $h = 10^{-2},10^{-3}$, $T = 100$, $d = 1$ are given in Figure~\ref{fig:conv_weights}. As reported in the caption of Figure~\ref{fig:conv_weights} and the legends, for a target accuracy of 3 digits in the computation of the CQ weights we need $n_0=8$ and $N_Q=28$ for $h=10^{-2}$, and $n_0=20$, $N_Q=48$ for $h=10^{-3}$. Thus the storage in our algorithm will be reduced from $N=10^4$  to $8+28$ quantities if $h=10^{-2}$, and from $N=10^5$ to $20+48$ if $h=10^{-3}$. For a more stringent target accuracy of six digits in the computation of the CQ weights, we need instead 15+48 quantities in storage (rather than $10^4$) for $h=10^{-2}$, and 27+74 for $h=10^{-3}$ (rather than $10^5$).

We can see that the chosen tolerance is not exceeded and the error for $n$ close to $n_0$ is close to the tolerance. However, the error for larger $n$ is much better than the tolerance. The truncation error and the quadrature error on the intervals away from the singularity all get quickly better with the larger $n$ so this is not so surprising. However, as the error bound for the Gauss-Jacobi quadrature in the first interval does not improve with increasing $n$, this suggests that possibly our estimate  in the first interval is not optimal.
% \red{I would rather think than our estimates for the Gauss-Legendre quadrature are not optimal, since the actual error is smaller than the target accuracy by almost three orders of magnitude. We could use in principle much less quadrature nodes on these intervals. \blue{I'm not sure that's the case. The difference between the GJ and GL intervals is that the error in the GJ interval increases with $t$, whereas the error in the Gauss Legendre intervals decreases with $t$. As we use the same quadrature rule for all times, it's unavoidable that the the error in the Gauss-Legendre intervals will be much smaller than tolerance for large $t$. But our estimates suggest that GJ error would be largest for large $t$, so the fact that we have low error for large $t$ suggests GJ estimate is not optimal.} }
This is not a great issue, as this single quadrature does not contribute a great deal to the overall costs: around 5\% of quadrature points are in this interval in the above calculations.

\begin{figure}
  \centering
  \includegraphics[width=.48\textwidth]{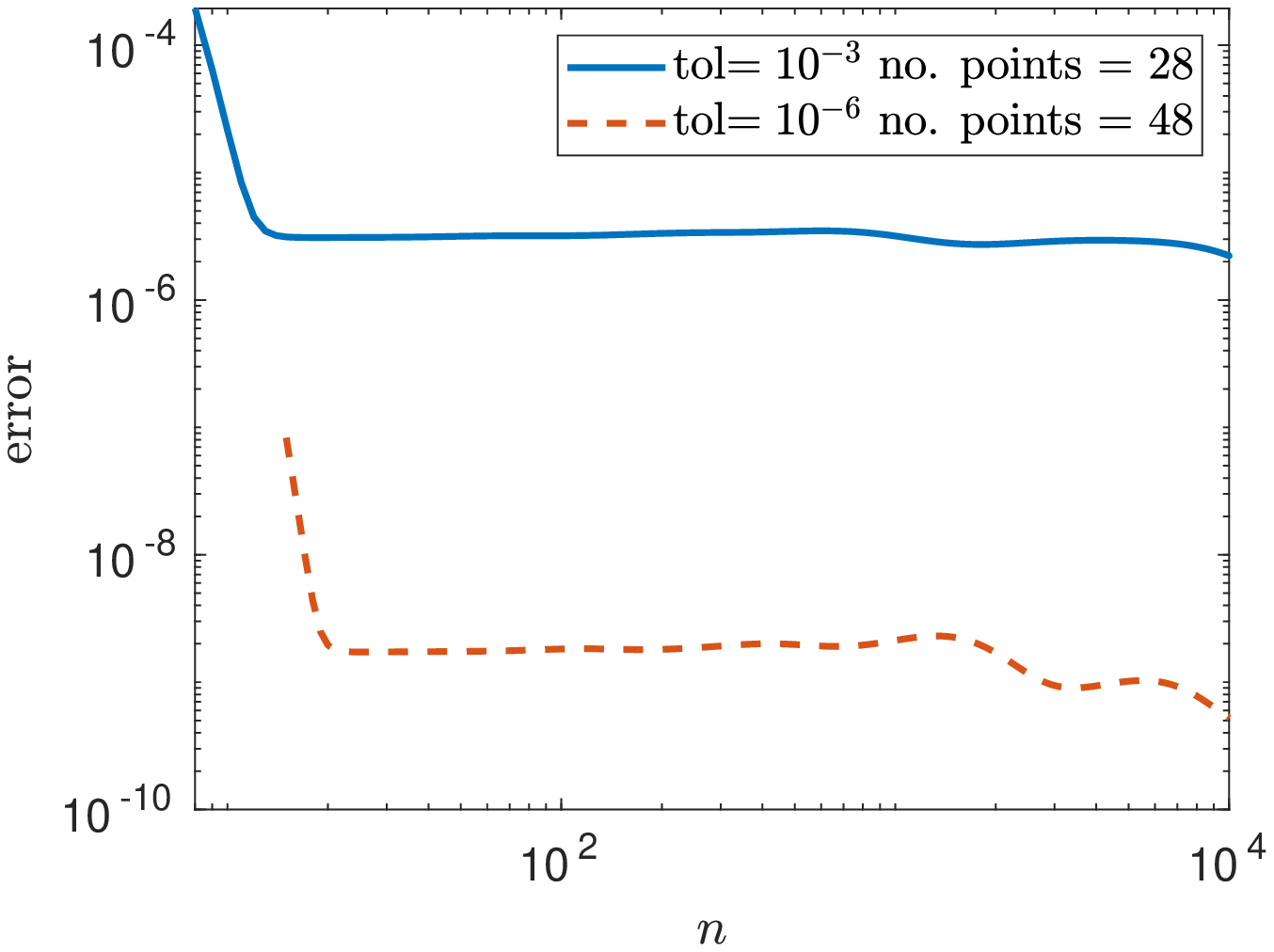}
    \includegraphics[width=.48\textwidth]{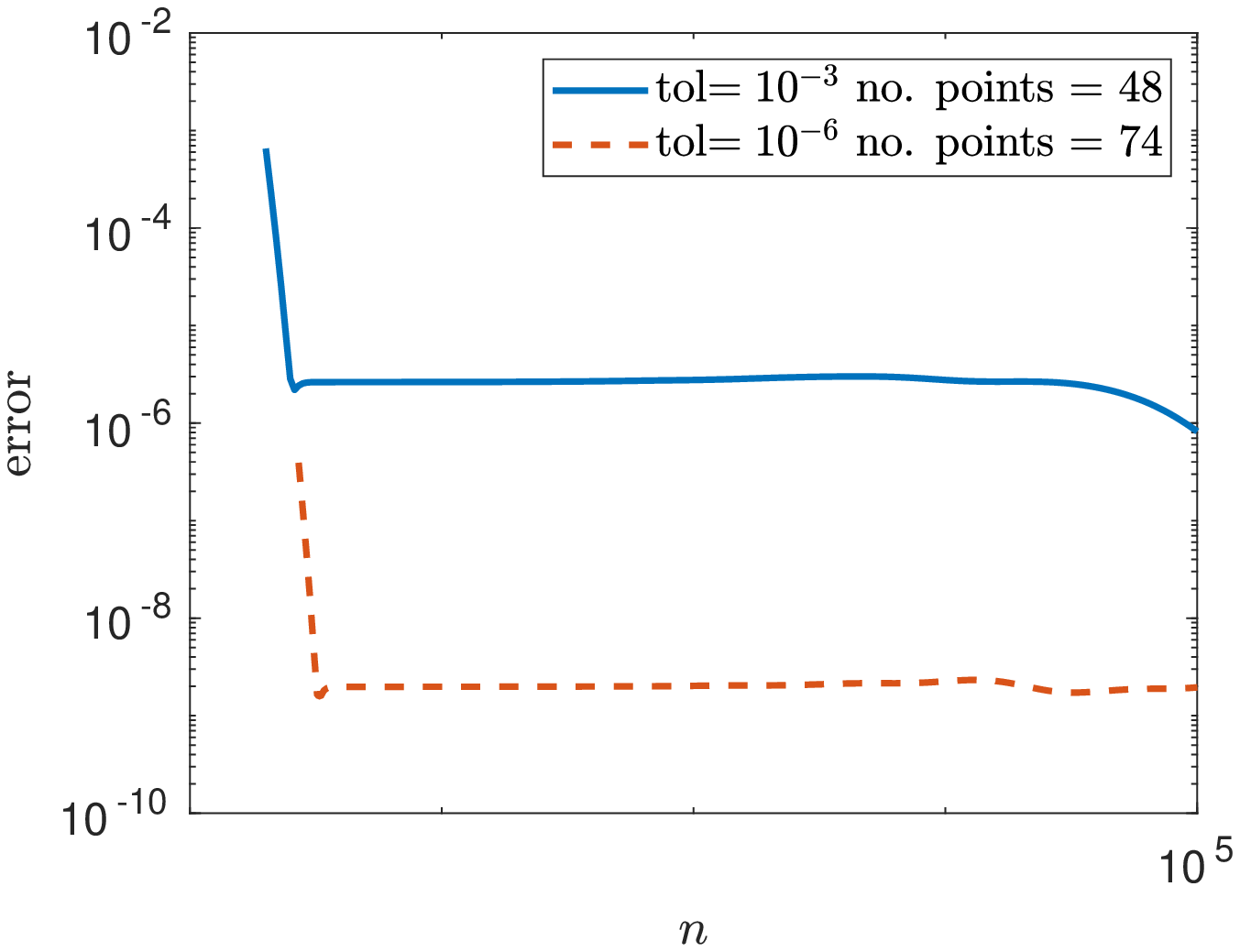}
  \caption{We plot the error $\left\|  \omegav_n(x)-\Iv_n \right\|$ for different tolerances and $n \geq n_0$ and $h = 10^{-2}$ on the left and with $h = 10^{-3}$ on the right. The total number of quadrature points is shown on the graphs. For $h = 10^{-2}$, $n_0$ was computed as 8 and 15 for the two tolerances and for $h = 10^{-3}$ these were 20 and 27.}
  \label{fig:conv_weights}
\end{figure}

\subsection{A linear Schr\"odinger equation with concentrated potential}\label{sec:lin_sch}

We start by considering the same example as in \cite{Lub92} and compute the solution $\psi(x,t)$, for $x\in\bR$, $t>0$ to:
\begin{equation}\label{lins}
\frac1\opi \partial_t \psi = \partial_{xx} \psi - \sum_{j=1}^{M} V_j(t) \delta_{x_j}  \psi ; \quad \psi(x,0)=\psi_0(x)
\end{equation}
for some given time-dependent amplitudes $V_j(t)$.

Assuming that
$$
V_j(t)\equiv \overline{V}_j \quad \mbox{ and } \quad \psi(x,t)=\psi_0(x)\e^{\opi \omega t}, \quad \mbox{for }\  t\le 0,
$$
the values $q_j(t):=\psi(x_j,t)$ satisfy the system of Volterra integral equations
\begin{equation}\label{chargeeq_lin}
q_j(t) + \sum_{k=1}^M  \int_0^t k(t-s, x_j-x_{k})\left( V_{k}(s) q_{k}(s)- \overline{V}_{k} \psi_0(x_{k})\e^{\opi \omega s} \right)   \,ds = \psi_0(x_j) \e^{\opi \omega t},
\end{equation}
for $j=1,\dots,M$ and the solution to \eqref{lins} can be written
\begin{equation}\label{bilins}
\psi(x,t) + \sum_{k=1}^M  \int_0^t k(t-s, x-x_{k}) \left( V_{k}(s) q_{k}(s)- \overline{V}_{k} \psi_0(x_{k})\e^{\opi \omega s} \right)   \,ds = \psi_0(x) \e^{\opi \omega t},
\end{equation}
with $k$ as in \eqref{green_schrodinger} with $D=1$. For more detail on derivation of this system see \cite{Lub92}.

%In order to solve \eqref{chargeeq_lin} by applying RK-CQ we change variables to
%\begin{equation}\label{newvar}
%\phi_j(t):=V_j(t)q_j(t)-\overline{V}_j \psi_0(x_j)\e^{\opi \omega t},
%\end{equation}
%in such a way that $\phi_j(0)=0$ and the equivalent system to \eqref{chargeeq_lin} is obtained
%\begin{equation}\label{sysphi}
%\phi_j(t) + V_j(t)\sum_{k=1}^M  \int_0^t k(t-s, x_j-x_{k})\phi_k(s)   \,ds = (V_j(t)-\overline{V_j})\psi_0(x_j) \e^{\opi \omega t},\ j=1,\dots,M.
%\end{equation}
%{\em This change of variables cannot be done if $V_j(t)=0$ for some $t$ or $j$. Thus I corrected the expressions in the example of Lubich so that this does not happen. Otherwise we have the problem of the starting values...}
%
%The application of Runge--Kutta based Convolution Quadrature to \eqref{sysphi} yields for $j=1,\dots, M$ approximations $\phiv_{j,n} \approx \phi_j(\tv_n)$, with $\tv_n=(t_{n}+c_i h)_{i=1}^s$, $n=0,1,\dots,N$, defined by
%\begin{equation}\label{dischargelins}
%\phiv_{j,n} + \Vv_{j,n}\sum_{k=1}^M \sum_{\ell=0}^{n} \Wv_{n-\ell}(\|x_j-x_{k}\|)\phiv_{k,\ell} = (\Vv_{j,n}- \overline{V_j})\psi_0(x_j) \e^{\opi \omega \tv_n}.
%\end{equation}
%where $\Vv_{k,\ell}:={\rm diag}(V_k(t_{\ell} +c_1 h),\dots,V_k(t_{\ell}+c_s h))$, for $k=1,\dots,M$.
%
%
For our experiments we take $M=2$,
\begin{equation}\label{data_ex1}
\begin{array}{l}
x_1=-1, \quad x_2=1, \quad \omega=1, \\[1em]
\psi_0(x)=\left\{ \begin{array}{ll} \cosh x/\cosh 1, \quad & \mbox{ for } |x|\le 1, \\[.5em]
\e^{1-|x|},  \quad & \mbox{ for } |x|> 1,
 \end{array} \right.\\[1.5em]
\overline{V}_1= \overline{V}_2=-c,\quad \mbox{ with }\ c=1+ \tanh 1, \\[1em]
V_1(t)=-c(1+\sin(t)),\  V_2(t)=-c(1-\sin(t)), \quad \mbox{for } t>0.%\\[1em]
%\red{V_1(t)=-c(1+0.6\sin(t)),\  V_2(t)=-c(1-0.6\sin(t)), \quad \mbox{for } t>0.}
\end{array}
\end{equation}
%
%We then need to solve the discrete system
%\begin{equation}\label{syslins}
%\begin{array}{l}
%\phiv_{1,n} + \Vv_{1,n}\sum_{\ell=0}^{n} \left( \Wv_{n-\ell}(0)\phiv_{1,\ell}
% + \Wv_{n-\ell}(2)\phiv_{2,\ell}\right)
% = (\Vv_{1,n}-\overline{V}_1)\psi_0(x_1) \e^{\opi \omega \tv_n} \\[1em]
%\phiv_{2,n} +  \Vv_{2,n}\sum_{\ell=0}^{n}\left( \Wv_{n-\ell}(0)\phiv_{2,\ell}
% +  \Wv_{n-\ell}(2) \phiv_{1,\ell}  \right)= (\Vv_{2,n}-\overline{V}_2)\psi_0(x_2) \e^{\opi \omega \tv_n},
%\end{array}
%\end{equation}
%This is
%\begin{align*}
%& \left(
%\begin{array}{ll}
%\Iv_s + \Vv_{1,n}\Wv_0(0) & \Vv_{1,n}\Wv_0(2)  \\
%\Vv_{2,n}\Wv_0(2) & \Iv_s +\Vv_{2,n}\Wv_0(0)
%\end{array}
%\right)
%\left( \begin{array}{l}
%\phiv_{1,n} \\
%\phiv_{2,n}
%\end{array}
%\right) \\
%& \ = \left(\begin{array}{l}
%(\Vv_{1,n}-\overline{V_1} )\psi_0(x_1)\e^{\opi \omega \tv_n}  \\
%(\Vv_{2,n}-\overline{V_2} )\psi_0(x_2)\e^{\opi \omega \tv_n}
%\end{array}
% \right) -  \sum_{\ell=0}^{n-1}
%\left(
%\begin{array}{ll}
%\Vv_{1,n}\Wv_{n-\ell}(0) & \Vv_{1,n}\Wv_{n-\ell}(2) \\
%\Vv_{2,n}\Wv_{n-\ell}(2) & \Vv_{2,n}\Wv_{n-\ell}(0)
%\end{array}
%\right)
%\left( \begin{array}{l}
%\phiv_{1,\ell}
%\\
%\phiv_{2,\ell}
%\end{array}
%\right)
%\end{align*}
%
%\red{Without changing variables to the $\phi_j$}

The application of Runge--Kutta based CQ to \eqref{chargeeq_lin} yields for $j=1,\dots, M$ approximations $\qv_{j,n} \approx q_j(\tv_n)$, with $\tv_n=(t_{n}+c_i h)_{i=1}^s$, $n=0,1,\dots,N$, defined by
\begin{equation}\label{dischargelins}
\qv_{j,n} + \sum_{k=1}^M \sum_{\ell=0}^{n} \Wv_{n-\ell}(x_j-x_{k})\left( \Vv_{k,\ell}\qv_{k,\ell}- \overline{V}_{k} \psi_0(x_{k})\e^{\opi \omega \tv_{\ell}} \right)   = \psi_0(x_j) \e^{\opi \omega \tv_n},
\end{equation}
where $\Vv_{k,\ell}:={\rm diag}(V_k(t_{\ell} +c_1 h),\dots,V_k(t_{\ell}+c_s h))$.

We then need to solve the discrete linear system
\begin{equation*}
\begin{array}{l}
\qv_{1,n} + \sum_{\ell=0}^{n} \left( \Wv_{n-\ell}(0)\left( \Vv_{1,\ell} \qv_{1,\ell}- \overline{V}_{1} \psi_0(x_{1})\e^{\opi \omega \tv_{\ell}} \right)
 + \Wv_{n-\ell}(2)\left( \Vv_{2,\ell} \qv_{2,\ell}- \overline{V}_{2} \psi_0(x_{2})\e^{\opi \omega \tv_{\ell}} \right)\right)
 \\[1em]
 \hfill = \psi_0(x_1) \e^{\opi \omega \tv_n} \\[1.5em]
\qv_{2,n} +  \sum_{\ell=0}^{n}\left( \Wv_{n-\ell}(0)\left( \Vv_{2,\ell} \qv_{2,\ell}- \overline{V}_{2} \psi_0(x_{2})\e^{\opi \omega \tv_{\ell}} \right)
 +  \Wv_{n-\ell}(2)\left( \Vv_{1,\ell} \qv_{1,\ell}- \overline{V}_{1} \psi_0(x_{1})\e^{\opi \omega \tv_{\ell}} \right)
 \right) \\[1em]
\hfill  = \psi_0(x_2) \e^{\opi \omega \tv_n},
\end{array}
\end{equation*}
This is
\begin{equation}\label{syslins}
\begin{array}{l}
\left(
\begin{array}{ll}
\Iv_s +\Wv_0(0)\Vv_{1,n} & \Wv_0(2) \Vv_{2,n} \\
\Wv_0(2)\Vv_{1,n} & \Iv_s +\Wv_0(0)\Vv_{2,n}
\end{array}
\right)
\left( \begin{array}{l}
\qv_{1,n} \\
\qv_{2,n}
\end{array}
\right) \\[1em]
\hspace{1em} = \left(\begin{array}{l}
(\Iv_s+\overline{V_1} \Wv_0(0)) \psi_0(x_1)\e^{\opi \omega \tv_n}   + \overline{V_2}\psi_0(x_2) \Wv_0(2)\e^{\opi \omega \tv_n} \\
(\Iv_s+\overline{V_2}\Wv_0(0))\psi_0(x_2)\e^{\opi \omega \tv_n}   +  \overline{V_1}\psi_0(x_1) \Wv_0(2)\e^{\opi \omega \tv_n}
\end{array}
 \right) \\[1em]
  \hspace{1em} \displaystyle -  \sum_{\ell=1}^{n}
\left(
\begin{array}{ll}
\Wv_{\ell}(0) & \Wv_{\ell}(2) \\
\Wv_{\ell}(2) & \Wv_{\ell}(0)
\end{array}
\right)
\left( \begin{array}{l}
\Vv_{1,n-\ell} \qv_{1,n-\ell}- \overline{V}_{1} \psi_0(x_{1}) \e^{\opi \omega \tv_{n-\ell}}
\\
\Vv_{2,n-\ell} \qv_{2,n-\ell}- \overline{V}_{2} \psi_0(x_{2})\e^{\opi \omega \tv_{n-\ell}}
\end{array}
\right)
\end{array}
\end{equation}

To compute the memory term on the right-hand side we use the fast method described in Section~\ref{sec:param}. The optimal choice of the various Gauss quadrature weights and nodes used in this algorithm depend on $d$ which is above $d =0$ or $d = 2$. However from our analysis it follows that the error when using the optimal choice for a certain $d=d_{\max}$ will be bounded by the tolerance  for every $d\in(0,d_{\max})$. We thus compute the optimal quadrature for $d=2$ and use it to represent both $\Wv_j(0)$ and $\Wv_j(2)$, for every $j=n_0+1,\dots,N$.

  \begin{figure}
    \centering
    \includegraphics[width=0.45\textwidth]{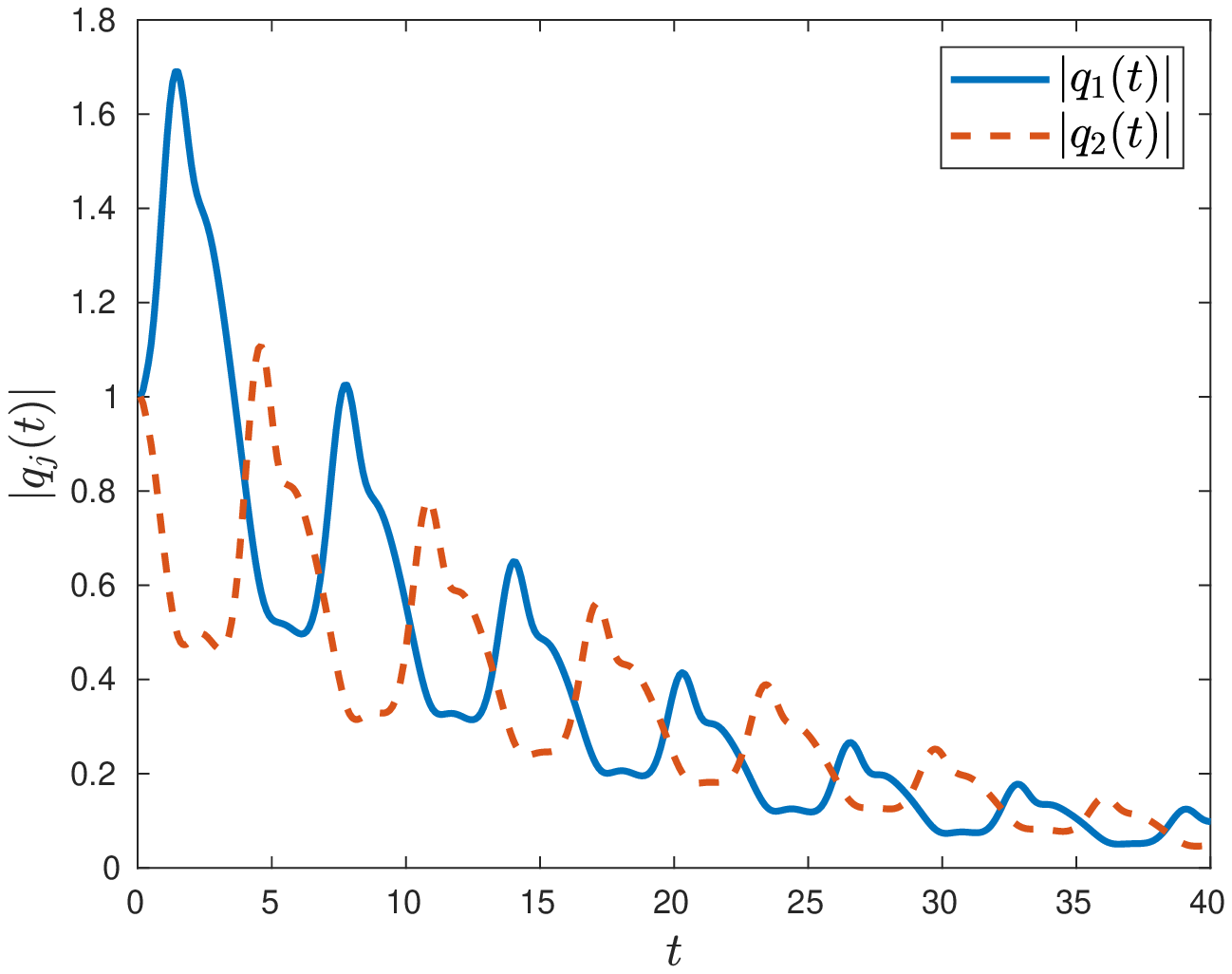}
        \includegraphics[width=0.45\textwidth]{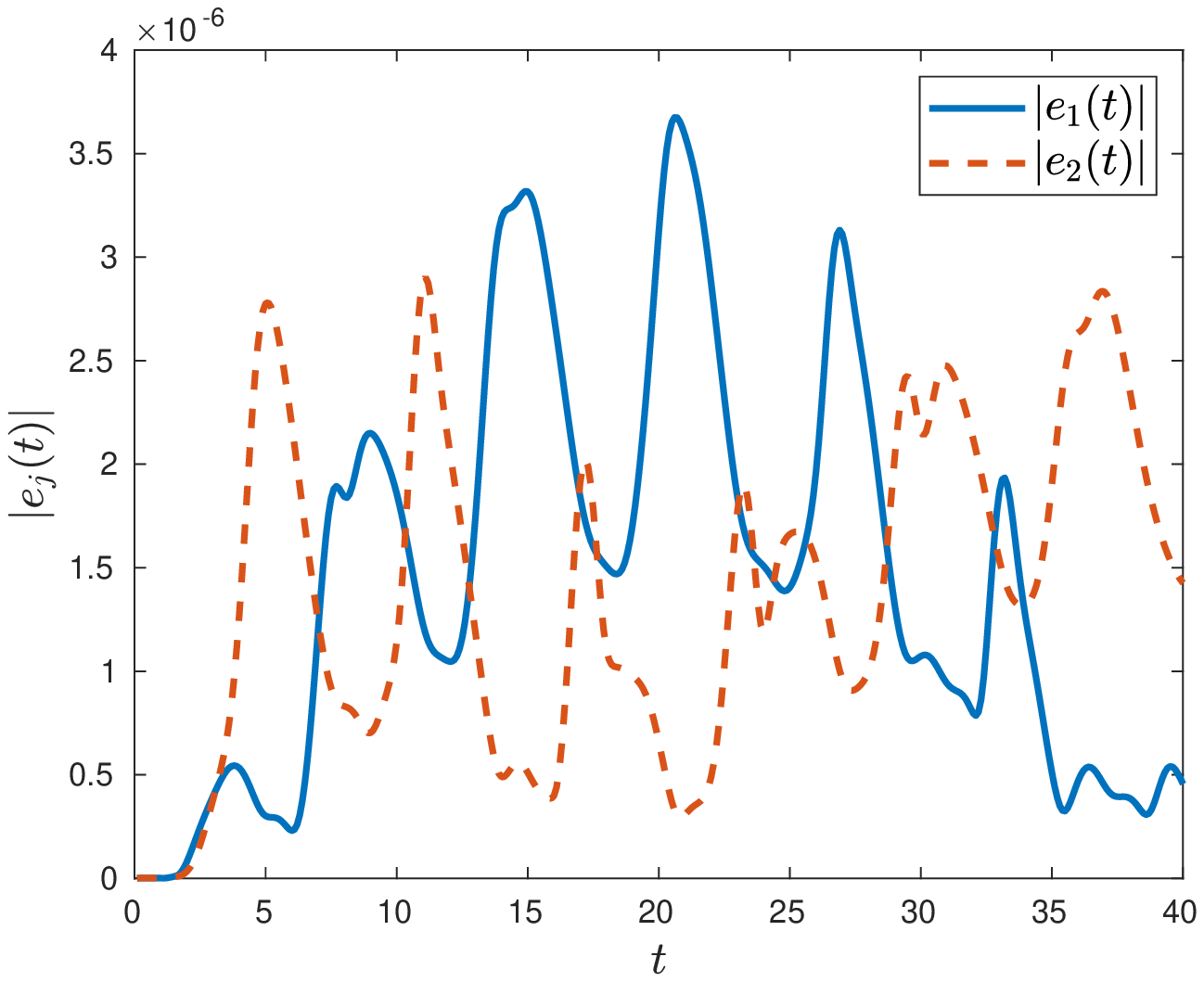}
    \caption{On the left we plot the approximate solution $|q_j(t)|$ of the linear charge equations \eqref{chargeeq_lin} obtained by the fast method with $\tol = 10^{-6}$. On the right we show the difference $e_j(t) = |q_j(t)-q_j^{\text{st}}(t)|$ between this solution and the one obtained by the standard implementation of CQ.}
    \label{fig:lin_sch_results}
  \end{figure}

  In Figure~\ref{fig:lin_sch_results} we show a plot of the solution and the difference in the solutions obtained by our fast method and the standard implementation of CQ. In these experiments we choose $T = 40$, $h = 0.1$, $\tol = 10^{-6}$ and $d = 2$. Our optimization routines returned $n_0 = 14$ and the total number of quadrature points $N_Q = 36$. As we see, the error is slightly larger than the tolerance, which is to be expected as the tolerance is valid for the computation of the weights not the final result of the discrete convolution.

\subsection{A non linear Schr\"odinger equation with concentrated potential}

We consider nonlinear Schr\"odinger equations with concentrated potentials and, in particular, the case studied in \cite{Negulescu}. Thus, we compute the solution $\psi(x,t)$, for $x\in\bR$, $t>0$ to:
\begin{equation}\label{ns2p}
\frac1\opi \partial_t \psi = \partial_{xx} \psi - \sum_{j=1}^{M} \gamma_j |\psi|^{2\sigma} \delta_{x_j}  \psi ; \quad \psi(0,x)=\psi_0(x),
\end{equation}
with $\gamma_j < 0, \ \sigma\ge 0$. Writing the solution using Duhamel's principle gives
\begin{equation}
  \label{eq:psi_form}
  \psi(t,x) = \phi(t,x)
  - \sum_{j = 1}^M\gamma_j\int_0^t k(t-s,x-x_j) |\psi(t-s, x_j)|^{2\sigma}\psi(t-s,x_j) ds,
\end{equation}
where $k$ is the Green's function \eqref{green_schrodinger} for $D = 1$ and $\phi(t,x)$ is the solution of the homogeneous problem:
\begin{equation}
  \label{eq:phi_hom}
\phi(t,x) =   \int_{\mathbb{R}} k(t,x-y) \psi_0(y) dy  .
\end{equation}
Evaluating \eqref{eq:psi_form} at $x = x_k$, $k = 1,\dots, M$,  gives the following system of integral equations
\begin{equation}
  \label{eq:psi_int_eq}
  \psi(t,x_k) +
   \sum_{j = 1}^M\gamma_j\int_0^t k(t-s,x_k-x_j) |\psi(t-s,x_j)|^{2\sigma}\psi(t-s,x_j) ds = \phi(t,x_k).
\end{equation}

%{\em Charge equations:}

In the following numerical experiments we set $M=2,\ x_1=-a,\ x_2=a,\ \gamma_1=\gamma_2=\gamma < 0$,
\begin{equation}\label{q12}
  q_1(t)=\psi(-a,t),\quad  q_2(t)=\psi(a,t), \quad
  \phi_1(t)=\phi(-a,t),\quad  \phi_2(t)=\phi(a,t),
\end{equation}
giving the system of equations
\begin{equation}\label{chargeeq_nonlin}
q_k(t) + \gamma \sum_{j=1}^2  \int_0^t k(t-s, x_k-x_{j})|q_{j}(s)|^{2\sigma} q_{j}(s)  \,ds = \phi_k(t), \quad k = 1,2.
\end{equation}

%For the numerical approximation we choose a semi-implicit version of the discrete system analagous to \eqref{syslins}, namely
% \begin{equation}\label{sysnonlins}
% \begin{array}{l}
% \left(
% \begin{array}{ll}
% \Iv_s +\gamma\Wv_0(0)\diag(|\qv_{1,n-1}|^{2\sigma}) & \gamma\Wv_0(6) \diag(|\qv_{2,n-1}|^{2\sigma}) \\
% \gamma\Wv_0(6)| \diag(|\qv_{1,n-1}|^{2\sigma}) & \Iv_s +\gamma\Wv_0(0) \diag(|\qv_{2,n-1}|^{2\sigma})
% \end{array}
% \right)
% \left( \begin{array}{l}
% \qv_{1,n} \\
% \qv_{2,n}
% \end{array}
% \right) \\[1em]
% \hspace{1em} \displaystyle = \left(\begin{array}{l}
% \phi_1(\tv_n) \\
% \phi_2(\tv_n)
% \end{array}
%  \right) -  \gamma\sum_{\ell=1}^{n}
% \left(
% \begin{array}{ll}
% \Wv_{\ell}(0) & \Wv_{\ell}(6) \\
% \Wv_{\ell}(6) & \Wv_{\ell}(0)
% \end{array}
% \right)
% \left( \begin{array}{l}
% \diag(|\qv_{1,n-\ell}|^{2\sigma} )\qv_{1,n-\ell}
% \\
% \diag(|\qv_{2,n-\ell}|^{2\sigma} )\qv_{2,n-\ell}
% \end{array}
% \right)
% \end{array}
% \end{equation}
% {\huge There is something to do with the initial data...}

For optimal performance, convolution quadrature requires that the data can be extended smoothly to negative times by zero.  Since $q_{1,2}(0)\ne 0$ we modify the system as follows
\begin{equation}\label{chargeeq_corr}
q_k(t) + \gamma \sum_{j=1}^2  \int_0^t k(t-s, x_k-x_{j})\left(|q_{j}(s)|^{2\sigma} q_{j}(s)-|q_{j}(0)|^{2\sigma} q_{j}(0)\right)  \,ds = \phi_k(t)-f_k(t),
\end{equation}
$k = 1,2$ with correction terms
\[
f_k(t) = \gamma \sum_{j=1}^2 |q_{j}(0)|^{2\sigma} q_{j}(0) \int_0^t k(s, x_k-x_{j})\,ds.
\]
To compute the correction terms we use the formula obtained using symbolic computation software
\begin{equation}
  \label{eq:corr_for}
\begin{aligned}
\int_0^t k(s,d)\,ds &= \frac{\e^{\opi \pi/4}}{\sqrt{\pi}} \int_0^{t} \frac{1}{\sqrt{4s}}\e^{\opi d^2/4s}\,ds \\
&=\frac{\e^{\opi \pi/4}}{4\sqrt{\pi}} \int_0^{4t} \frac{\e^{\opi d^2/u}}{\sqrt{u}}\,du = \frac{\e^{\opi \pi/4}}{2\sqrt{\pi}} \int_0^{2\sqrt{t}} \e^{\opi d^2/\xi^2}\,d\xi \\
&= e^{\mi \pi/4} e^{\mi \frac{d^2}{4t}} \sqrt{t/\pi}-(d/2)
\erf\left(\frac{e^{3\pi \mi/4}d}{2\sqrt{t}}\right) -d/2,
      \end{aligned}
\end{equation}
where $\erf$ is the error function.

  After discretization of \eqref{chargeeq_corr} using CQ as in Section~\ref{sec:lin_sch} we obtain a non-linear system to be solved at each step of the form
  \[
    \begin{split}
    \begin{pmatrix}
\Iv_s +\gamma\Wv_0(0) & \gamma\Wv_0(6)  \\
\gamma\Wv_0(6)|  & \Iv_s +\gamma\Wv_0(0)
\end{pmatrix}
&
\begin{pmatrix}
\diag(|\qv_{1,n}|^{2\sigma})\qv_{1,n}-\diag(|\qv_{1,0}|^{2\sigma})\qv_{1,0} \\
\diag(|\qv_{2,n}|^{2\sigma})\qv_{2,n}-\diag(|\qv_{2,0}|^{2\sigma})\qv_{2,0}
\end{pmatrix}\\
&=
\begin{pmatrix}
  \phi_1(\tv_n) \\
\phi_2(\tv_n)
\end{pmatrix}
+ \mathbf{H}(\tv_n),
    \end{split}
  \]
  where $\mathbf{H}(t)$ is the history term containing terms known at time-step $n$. We solve the non-linear equation by a fixed-point iteration with the initial guess given by the solution at the previous time-step. The history term $\mathbf{H}(\tv_n)$ is computed using our fast method.

\subsection{The linear case}

Let us first consider the linear case, i.e., $\sigma = 0$. We look for solutions of the form
\begin{equation}
  \label{eq:psi_lin}
\psi(t,x) = \alpha e^{\mi \lambda_\tf t} \phi_\tf(x)+ \beta e^{\mi \lambda_\te t} \phi_\te(x),
\end{equation}
with $\alpha^2+\beta^2 = 1$ and
\begin{equation}
  \label{eq:phi_f}
  \phi_\tf(x) = N_\tf \left(K(\mi \lambda_\tf,x+a)+K(\mi \lambda_\tf,x-a)\right)
\end{equation}
and
\begin{equation}
  \label{eq:phi_e}
  \phi_\te(x) = N_\te \left(K(\mi \lambda_\te,x+a)-K(\mi \lambda_\te,x-a)\right),
\end{equation}
where constants $N_\tf$ and $N_\te$ are chosen so that $\|\phi_\tf\|_{L^2(\mathbb{R})} = \|\phi_\te\|_{L^2(\mathbb{R})} = 1$. Here, $\phi_\tf$ and $\phi_\te$ correspond to the fundamental respectively excited state, see \cite{Negulescu}.
The initial data is hence of the form
\begin{equation}\label{ininegu}
\psi_0(x)=\alpha \phi_f(x) + \beta\phi_e(x).
\end{equation}

Substituting this ansatz into \eqref{ns2p} with $\sigma = 0$,  gives the following relations that need to be satisfied by $\lambda_\tf$ and $\lambda_\te$
\[
K(\mi\lambda_{\tf},0)+K(\mi\lambda_{\tf},2a) = -\frac1\gamma
\]
and
\[
K(\mi\lambda_{\te},0)-K(\mi\lambda_{\te},2a) = -\frac1\gamma.
\]

We choose the parameters $\alpha=\sqrt{0.01}$, $\beta=\sqrt{0.99}$, $a=3$, $\gamma=-0.5$. Solving numerically the above nonlinear equations gives the eigenvalues $\lambda_f = 0.085894322668323$ and $\lambda_e = 0.021229338264198$. For this special initial data, formulas similar to \eqref{eq:corr_for} are available for expressing the solution to the free Schr\"odinger equation $\phi(t,x)$ in terms of the error function, see \cite{Negulescu}.

% As computed in \cite{Negulescu}, the solution to the free Schr\"odinger equation with this initial data is analytically available:
% \begin{equation}\label{rhs}
% \begin{array}{l}
% \displaystyle \phi_1(t)=\frac{\alpha N_f}{2\pi}(I_A(\lambda_f,t)+I_B(\lambda_f,t))+\frac{\beta N_e}{2\pi}(I_A(\lambda_e,t)-I_B(\lambda_e,t))\\[1em]
% \displaystyle\phi_2(t)=\frac{\alpha N_f}{2\pi}(I_A(\lambda_f,t)+I_B(\lambda_f,t))-\frac{\beta N_e}{2\pi}(I_A(\lambda_e,t)-I_B(\lambda_e,t)),
% \end{array}
% \end{equation}
% with auxiliary functions
% \begin{equation}\label{IA}
% I_A(\lambda,t)=\frac{\pi}{\sqrt{\lambda}}\e^{\opi\lambda t}(1-\erf(\sqrt{\opi \lambda t}))
% \end{equation}
% and
% \begin{equation}\label{IB}
% I_B(\lambda,t)=\frac{\pi}{2\sqrt{\lambda}}\e^{\opi\lambda t} \left[\e^{2\sqrt{\lambda}a}\left(1-\erf \left(\sqrt{\opi \lambda t} + \frac{a}{\sqrt{\opi t}} \right)\right)+ \e^{-2\sqrt{\lambda}a} \left(1-\erf\left(\sqrt{\opi \lambda t} - \frac{a}{\sqrt{\opi t}} \right) \right)\right]
% \end{equation}
% To evaluate the error function $\erf$ on complex arguments we use the MATLAB function {\tt erfz} written by Marcel Leutenegger available in ... {\em I don't know how to quote this}.

We apply the fast convolution quadrature method based on the 2-stage Radau IIA Runge-Kutta method. The numerical results with $T = 100$ and $h = 1$ are given in Figure~\ref{fig:neg_lin}. There we plot $|q_j(t)|^2$ and the error
\[
  e_j(t) = \left||q_j(t)|^2-|q^{\text{ex}}_j(t)|^2\right|,
\]
where $q_j(t)$ is the numerically obtained solution and $q^{\text{ex}}_j(t)$ the exact solution. The results are of high quality with, as expected, larger error near $t = 0$ and slight increase in error with increase in time.

\begin{figure}
  \centering
  \includegraphics[width=.45\textwidth]{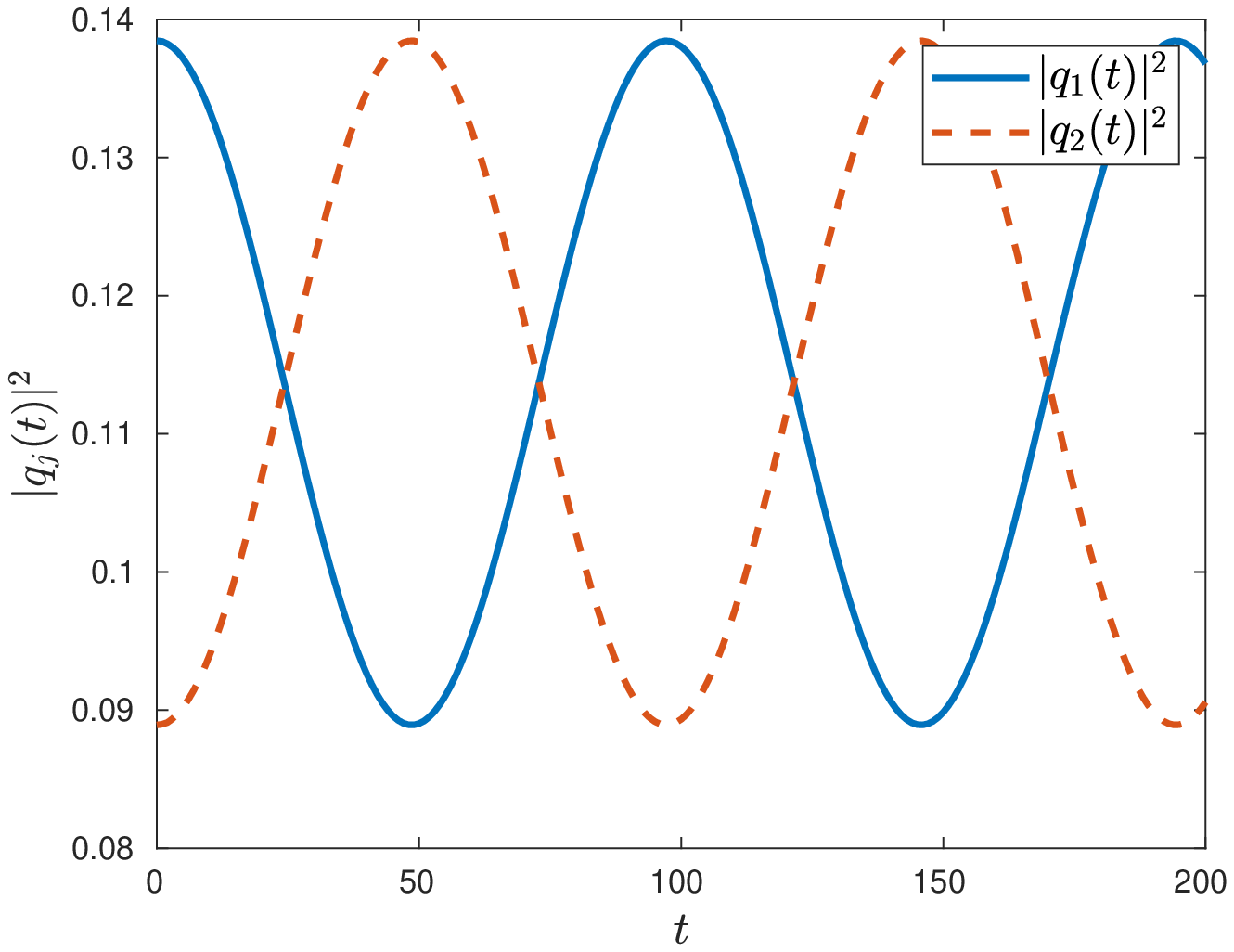}
    \includegraphics[width=.45\textwidth]{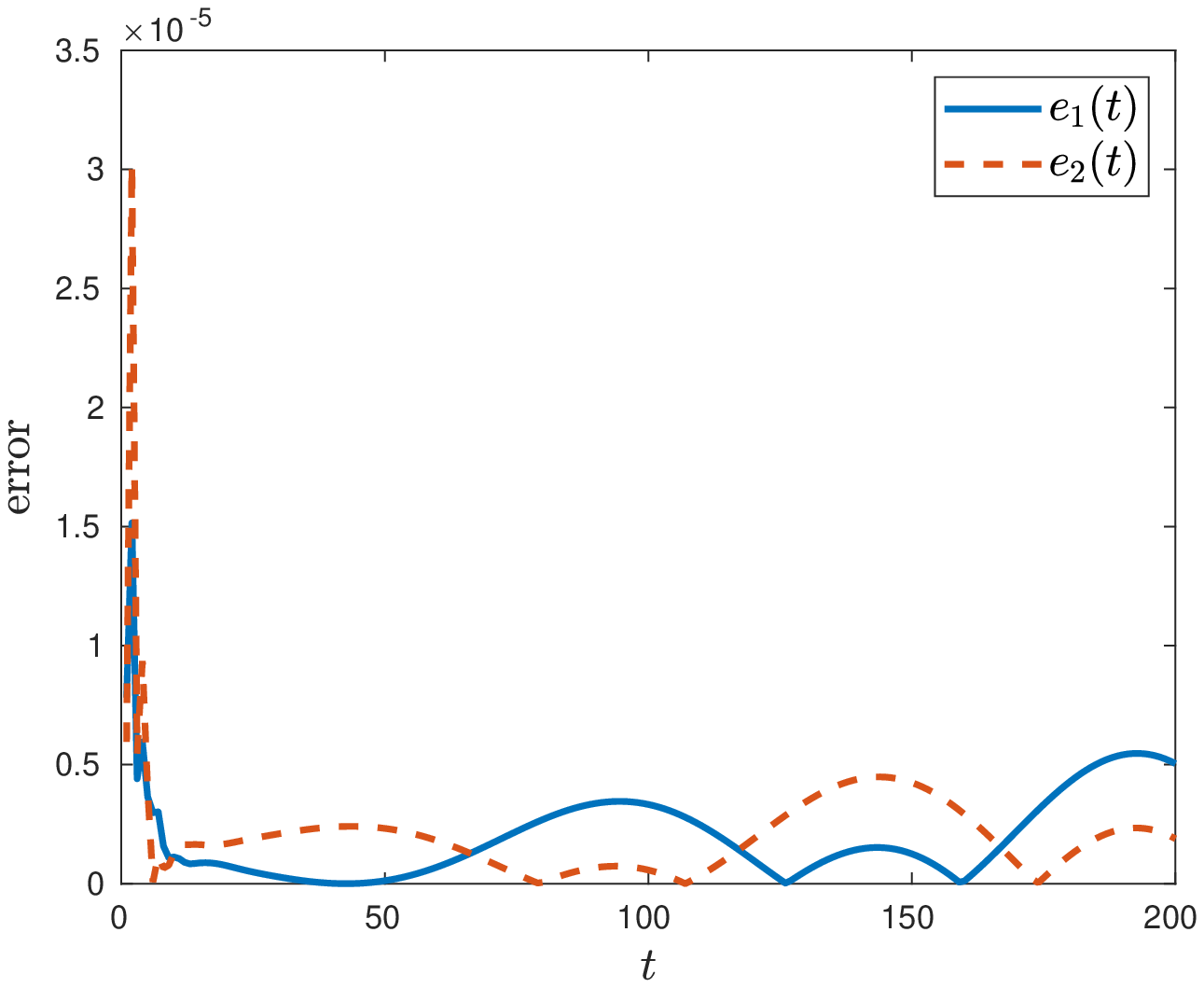}
  \caption{On the left we plot the numerically computed $|q_j(t)|^2$ for $j = 1,2$. On the right we show the error $e_j(t) = ||q_j^{\mathrm{ex}}(t)|^2-|q_j(t)|^2|$.}
  \label{fig:neg_lin}
\end{figure}

The importance of this example is that it shows the so-called beating motion of the system. This is not an easy problem to solve numerically and these excellent results in the linear case give us confidence in the nonlinear results presented in the next section. In Table~\ref{tab:eh} we also show the numerically observed convergence rate of the error measure
\[
  e^h = \max_n \max(e_1(t_n),e_2(t_n))
\]
as the time-step $h > 0 $ is reduced.  We should note here that there seems to be a slight discrepancy in the initial data we obtain to the ones that could be seen in Fig.~7 in \cite{Negulescu}, where the same parameters are used. We obtain $|q_1(0)|^2 \approx 0.1385$ and $|q_2(0)|^2 \approx 0.0889$, values slightly smaller than in \cite{Negulescu}. The reason for this is possibly different normalization.

\begin{table}
  \centering
  \begin{tabular}{|c|c|c|c|c|c|}
\hline    $h$ & 2& 1 & 1/2& 1/4& 1/8\\\hline
    $e^h$ &$7.22 \times 10^{-5}$&$3.11 \times 10^{-5}$&$1.12 \times 10^{-5}$&$5.69 \times 10^{-6}$&$2.29 \times 10^{-6}$\\
    EOC &  &1.2 & 1.5& 0.98 & 1.3\\\hline
  \end{tabular}
  \caption{Error $e_h$ for decreasing values of $h$ for the linear beating motion problem and the estimated order of convergence.}
  \label{tab:eh}
\end{table}

%\red{Should we use such a high order RK method in view of the low order of convergence achieved? It might make sense to see what happens with BE. \blue{It would make sense to see what happens with BE, but there may still be sense in using high-order methods in view of strong damping/dispersion properties of BE for wave problems} By the way I do not understand the second row of the table? From $h=1/2$ to $h=1/4$ there is a EOC of $0.98$. What does it mean? $h^{0.98}$? \blue{Yes, the error has approximatelly halved.}}

\subsection{Non-linear case}

In the non-linear case we take the same initial data \eqref{ininegu} as in the linear case and observe for which values of $\sigma > 0$ is the beating effect supressed. All other parameters are the same as in the linear case except for
\[
\gamma = -\frac{1}{|\psi_0(a)|^{2\sigma}+ |\psi_0(-a)|^{2\sigma}}.
\]

We compute the solution for $\sigma = 0.3, 0.6, 0.7, 0.8, 0.9, 0.98$. For low values of $\sigma$ the beating phenomenon is still visible, whereas for stronger non-linearities it begins to disappear.  It is also interesting that the numerical computation becomes increasingly difficult with increasing $\sigma$. This is not suprising as it is known that for large enough $\sigma$ blow-up can occur in finite time \cite{Negulescu}. Solutions for the different $\sigma$ is given in Fugures~\ref{fig:sig03}--\ref{fig:sig09} with some extra detail for $\sigma = 0.9$ given in Figure~\ref{fig:sig09_err}. The plots for $\sigma = 0.3, 0.6, 0.7, 0.8$ do not change at this scale for smaller time-step $h$. For $\sigma = 0.9$ the basic shape of the solution seems to be well captured but as indicated in Figure~\ref{fig:sig09_err} $q_1$ becomes increasingly oscillatory and the error for $q_1$ increases significantly for larger $t$. Finally for $\sigma = 0.98$ blow-up seems to occur near $t = 14.3$, decreasing $h$ just increases the height of the peak. Note that the largest computation for $\sigma = 0.9$ and $h = 1/128$ required us to compute $N = 25600$ time-steps. For this case our algorithm needed $n_0 = 57$ direct steps and $N_Q = 112$ quadrature nodes for $\tol = 10^{-8}$.
In Figure~\ref{fig:timings} we compare the computational times of the new method, the $O(N\log^2N)$ method based on FFT introduced in \cite{HaLuSch} and as modified in \cite{Ban10,lb_ms},  and the standard naive $O(N^2)$ implementation. These timings, clearly show that  the availability of a fast method was essential to perform experiments in reasonable time.  Furthermore the new method is the fastest in all the listed experiments, though the FFT based method is also very fast. However, only the new method brings savings in terms of memory. These are not significant in the one dimensional cases with a potential concentrated in only two points as investigated here. For a potential  concetrated in many points,  or for higher dimensional problems with the potential concentrated on a manifold, the memory savings will become equally important.

% \begin{table}
%   \centering
%   \begin{tabular}{|ccccc|}\hline
%     $N$ & $h$ & $\sigma$ & time (fast)& time (standard)\\\hline
%     400  & $1/2$ & $0.6$ & 3.1 & 5.6\\
%     800 & $1/4$ & $0.7$ & 6.6 & 16.3\\
%     1600 &  $1/8$ & $0.8$ & 12.9 & 54.7\\
%     25600 & $1/128$ & $0.9$ & 362 & 11429\\\hline
%   \end{tabular}
%   \caption{We solve the non-linear Schr\"{o}dinger equation with concentrated potentials using the new method and the standard $O(N^2)$ implementation. The columns give the number of time-steps $N$, the time-step $h$, the parameter controlling the nonlinearity $\sigma$, and the timings in seconds. Note that for the final computation while the fast method took 6 minutes, the standard one took more than 3 hours. }
%
%   \label{tab:timings}
% \end{table}

  \begin{figure}
    \centering
      \includegraphics[width=0.6\textwidth]{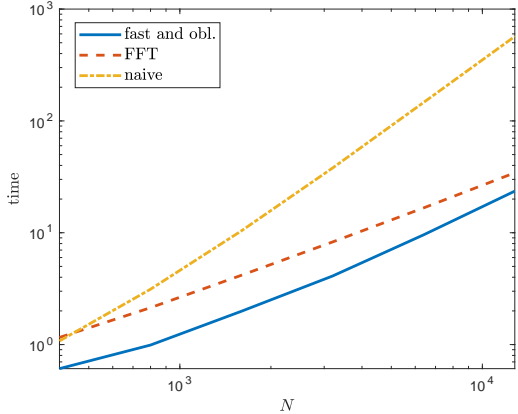}
    \caption{We compare the new, fast and oblivious method with the method of \cite{HaLuSch} based on FFT and the naive $O(N^2)$ implementation. To do this, we solve the non-linear Schr\"{o}dinger equation with concentrated potentials using the ones stage Radau IIA method, i.e., the backward Euler method, with a fixed final time $T = 200$, $\sigma = 0.8$ and an increasing number of time steps $N$. We see that the new method is the fastest, whereas the naive method is extremely slow.}
    \label{fig:timings}
  \end{figure}

\begin{figure}
  \centering
  \includegraphics[width=0.45\textwidth]{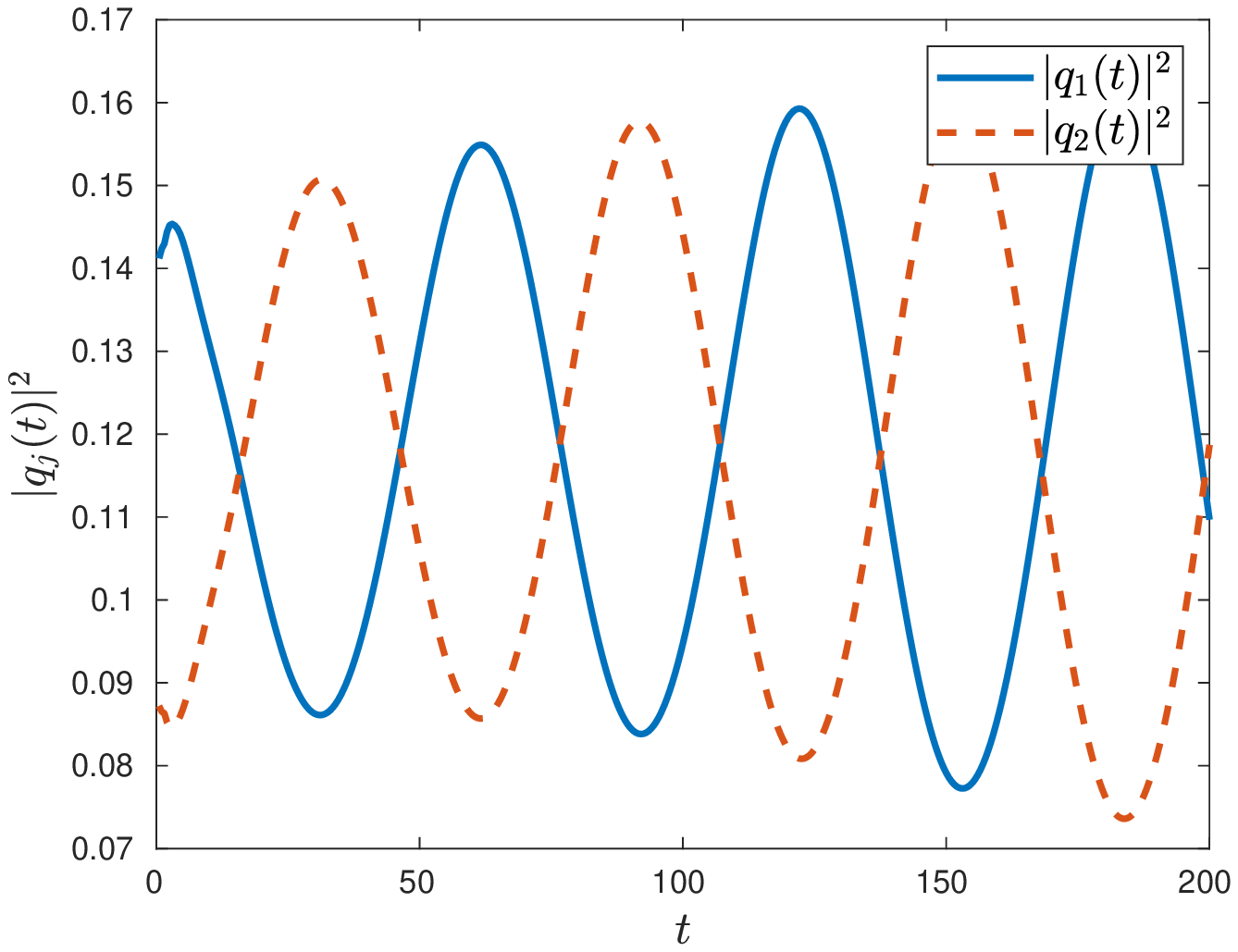}
    \includegraphics[width=0.45\textwidth]{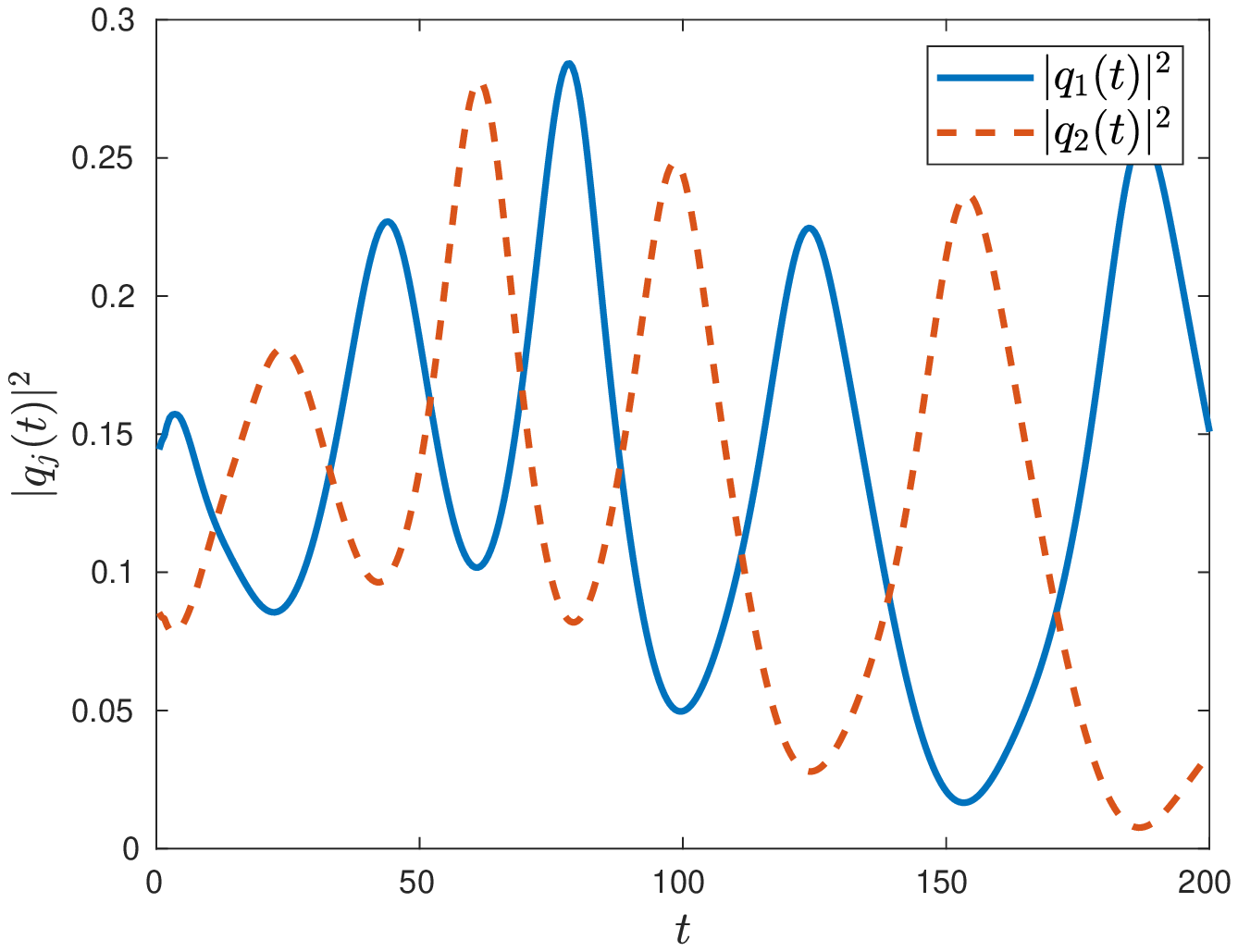}
  \caption{On the left we show the solution for $\sigma = 0.3$ and on the right for $\sigma = 0.6$. Time-step $h = 1/2$ is used in both calculations.}
  \label{fig:sig03}
\end{figure}

\begin{figure}
  \centering
  \includegraphics[width=0.45\textwidth]{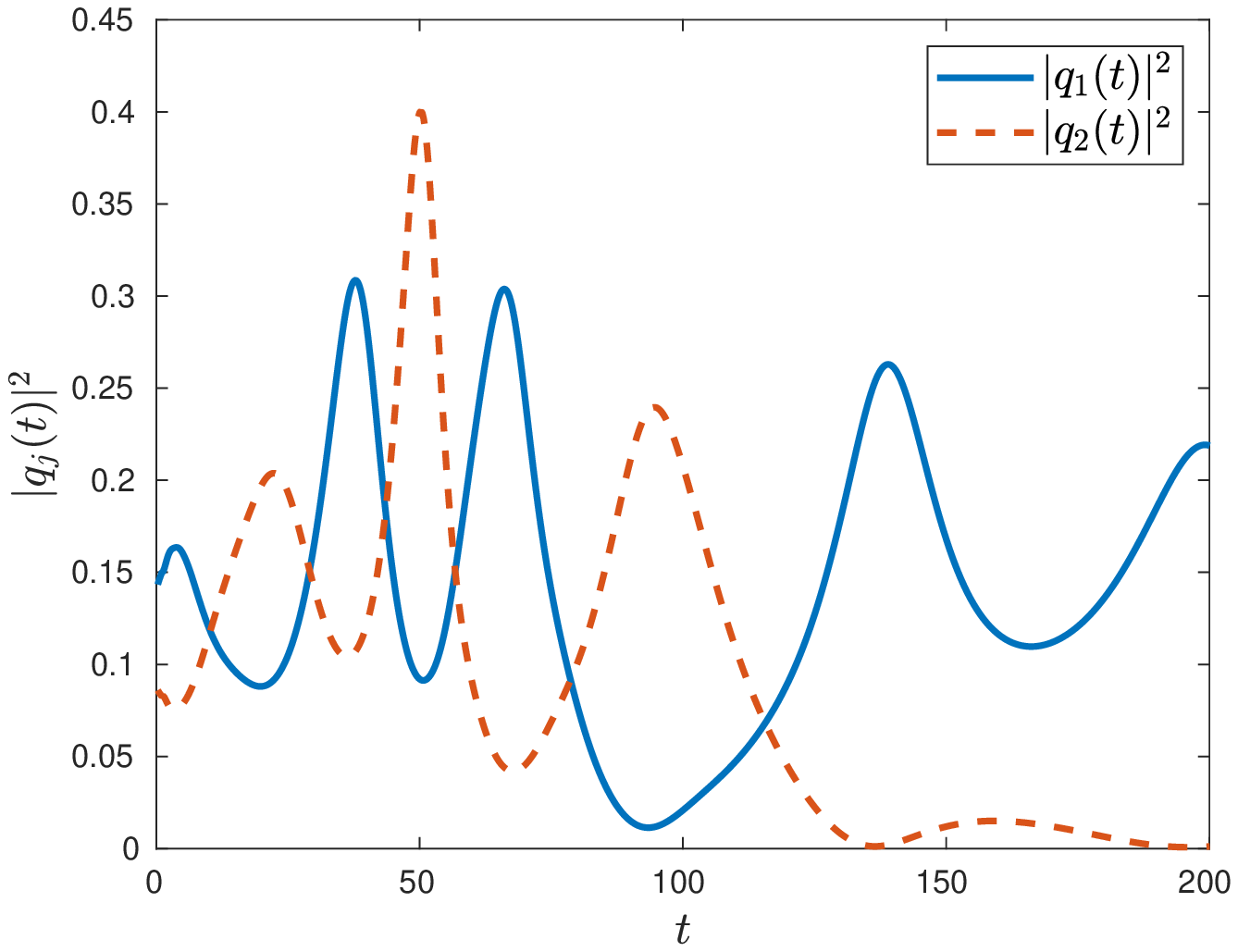}
    \includegraphics[width=0.45\textwidth]{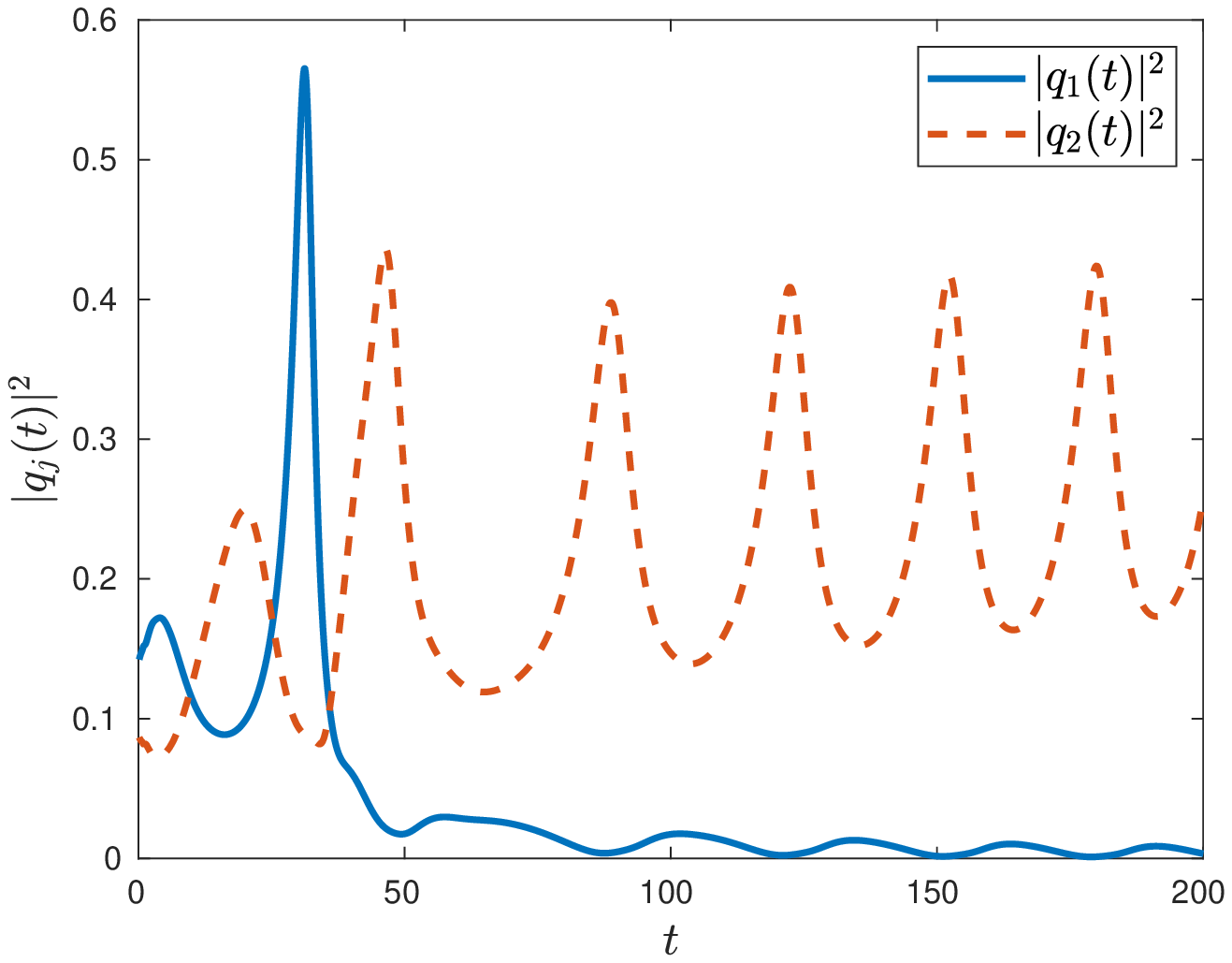}
  \caption{On the left we show the solution for $\sigma = 0.7$ and on the right with $\sigma = 0.8$. Time-step $h = 1/4$ is used for $\sigma = 0.7$ and $h = 1/8$ for $\sigma = 0.8$.}
  \label{fig:sig07}
\end{figure}

\begin{figure}
  \centering
  \includegraphics[width=0.45\textwidth]{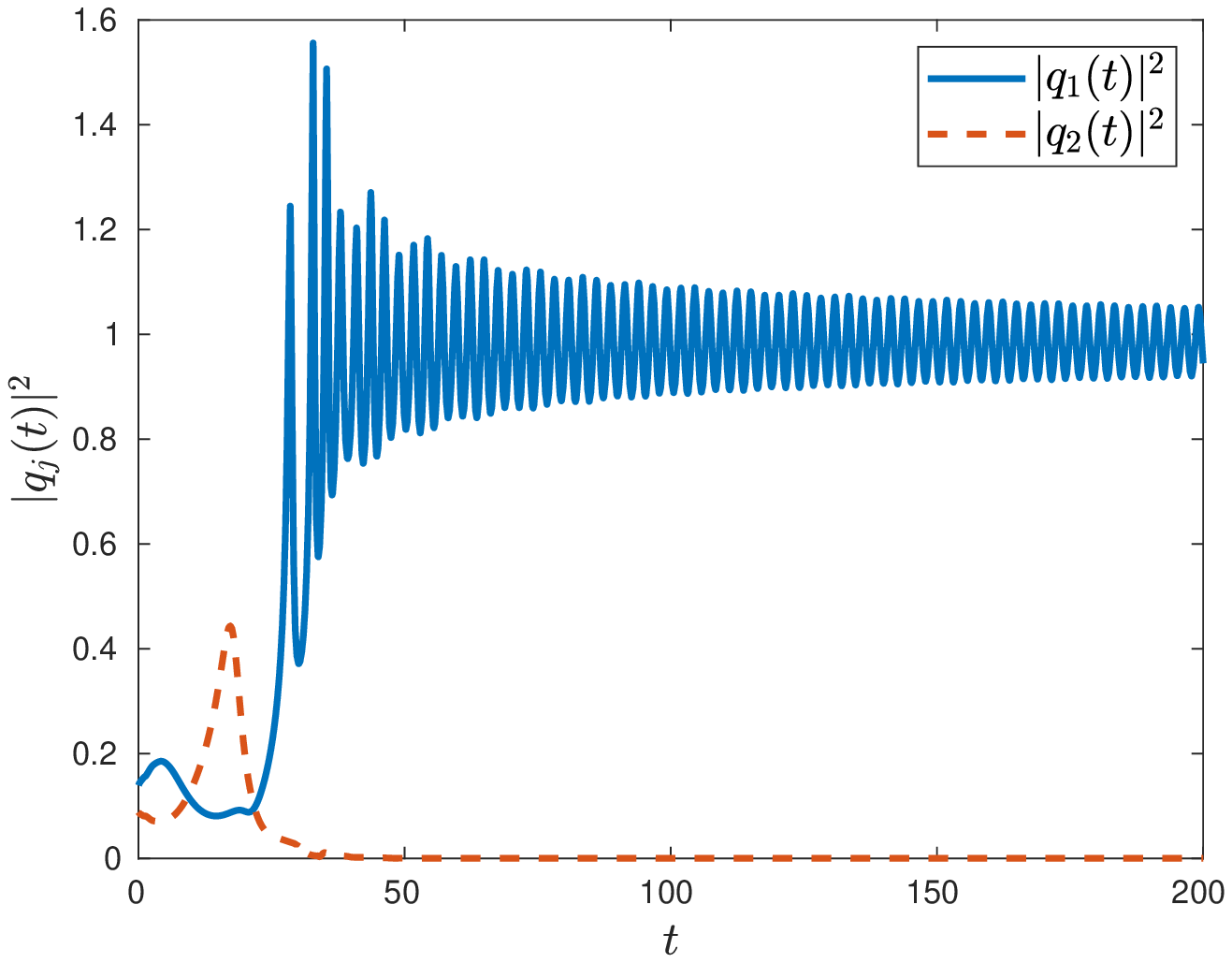}
    \includegraphics[width=0.45\textwidth]{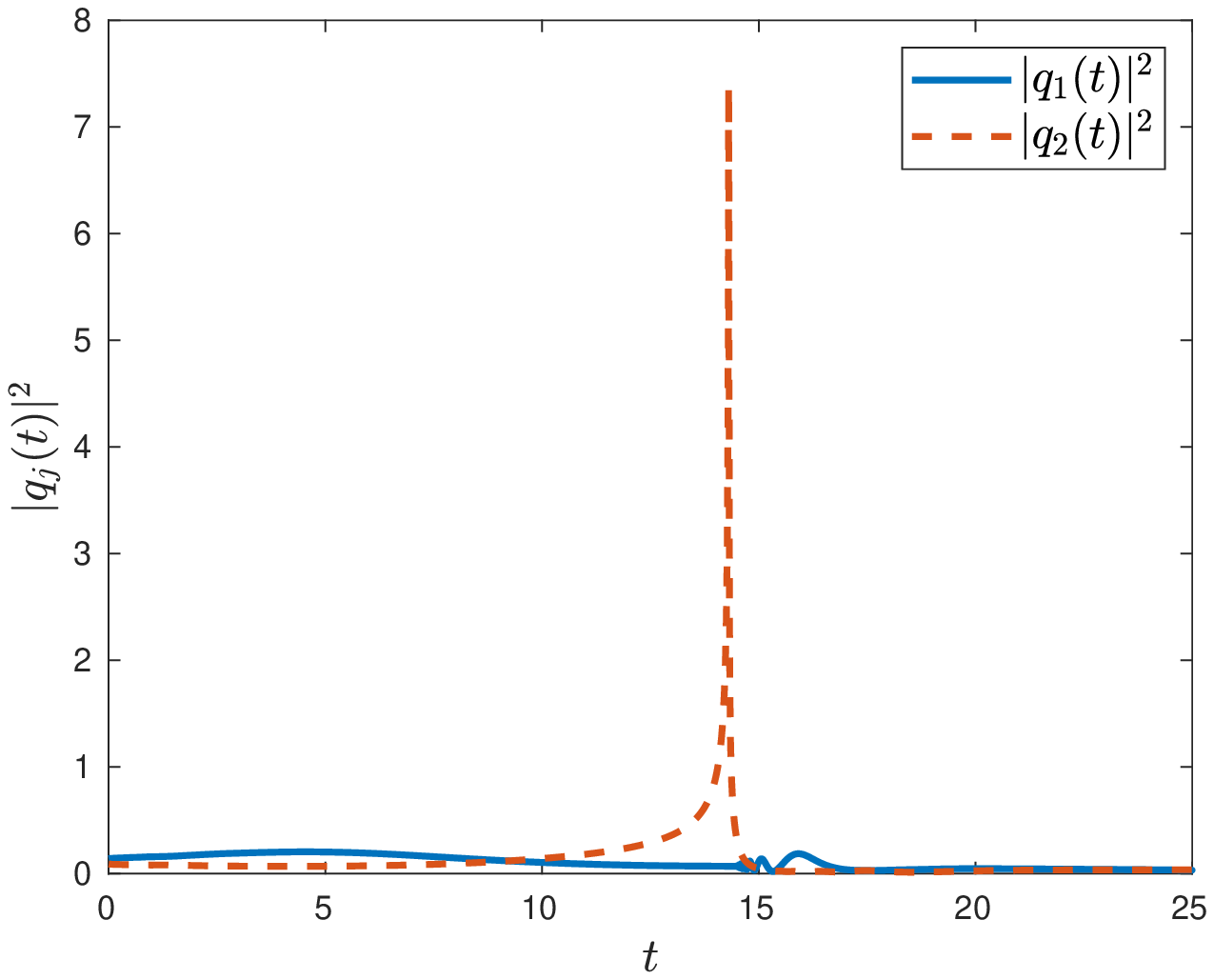}
  \caption{On the left we show the solution for $\sigma = 0.9$ and on the right with $\sigma = 0.98$. Time-step $h = 1/128$ is used for $\sigma = 0.9$ and $h = 1/256$ for $\sigma = 0.98$. On the right blow-up seems to occur around $t = 14.3$.}
  \label{fig:sig09}
\end{figure}

\begin{figure}
  \centering
  \includegraphics[width=0.6\textwidth]{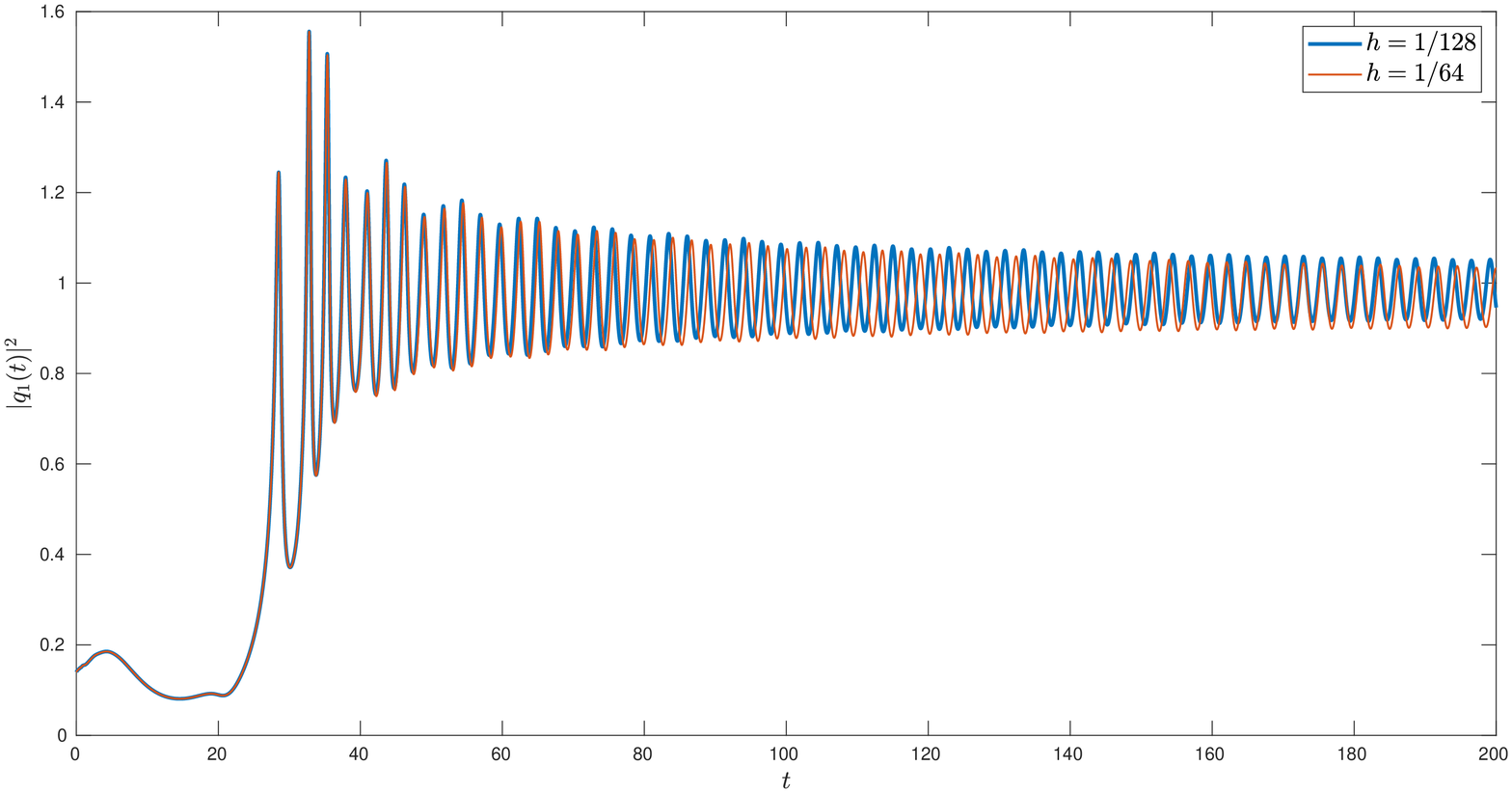}
    \includegraphics[width=0.38\textwidth]{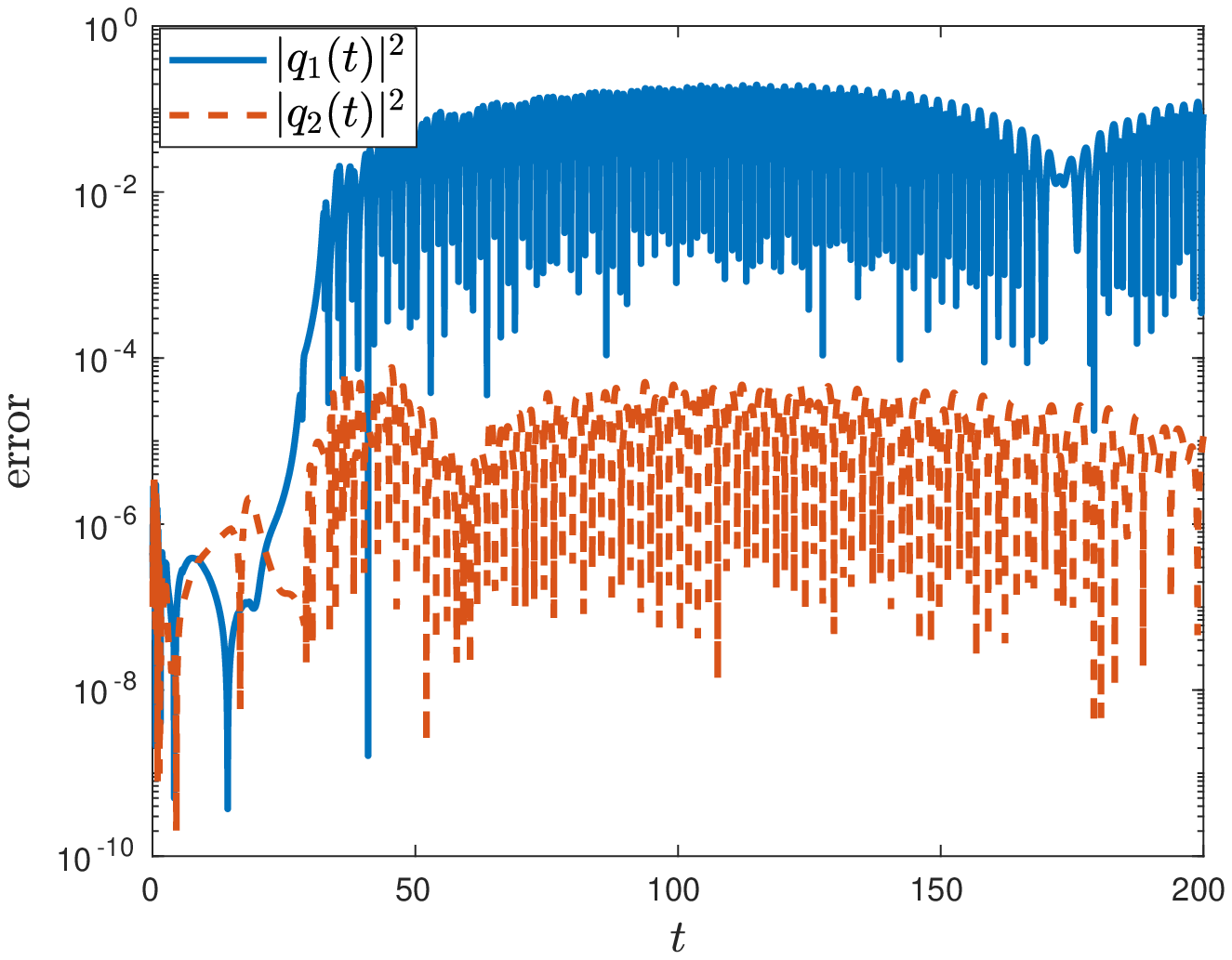}
  \caption{On the left we compare the solution $|q_1(t)|^2$ for $\sigma = 0.9$ and two choices of time-step: $h = 1/64$ and $h = 1/128$. On the right we show the difference in computing $|q_j(t)|^2$, $j = 1,2$, for the two different choices of $h$.}
  \label{fig:sig09_err}
\end{figure}

\def\cprime{$'$}

\section{Conclusions}
We have developed a special algorithm for the implementation of Lubich's Convolution Quadrature when applied to the integral formulation of Schr\"odinger equations with concentrated potential. The new algorithm belongs to the family of the so-called {\em fast and oblivious convolution algorithms}, since for the approximation of the solution at $N$ time steps it is able to reduce the complexity from $O(N^2)$ operations to $O(N\log N)$ and the storage from $O(N)$ to $O(n_0+\log N)$, with $n_0\ll N$. These features allow us to reliably simulate the behavior of the solution to non linear problems for long times and/or with a very small step, in order to capture high oscillations or finite time blow up. Our results are in good agreement with those reported in \cite{Lub92} and in \cite{Negulescu}, where two different methods with complexity $O(N^2)$ and memory requirements $O(N)$ are used. The MATLAB codes written to perform the simulations in the current paper can be found in \cite{codes}.

Future research will address the theoretical analysis of the error associated to the Convolution Quadrature approximation of relevant non linear cases, the control of the time step and the generalization of the algorithm and its application to the two dimensional case, following the recent results in \cite{CCF17}.

\section*{Acknowledgements}
The second author acknowledges Alessandro Teta for very useful discussion about the models during the preparation of the paper. The second author also acknowledges partial support by INdAM-GNCS and the Spanish grant MTM2016-75465-P.

%\bibliographystyle{abbrv}
%\bibliography{schrodinger}

\def\cprime{$'$}

\end{document}